\documentclass[11pt,A4paper]{article}

\usepackage{bbm}
\usepackage{hyperref}
\usepackage{amsmath}
\usepackage{amsfonts}
\usepackage{amsthm}
\usepackage{amssymb}
\usepackage{graphicx}
\usepackage[utf8]{inputenc}

\allowdisplaybreaks

\theoremstyle{remark}
\newtheorem{remark}{Remark}[section]

\theoremstyle{plain}
\newtheorem{theorem}[remark] {Theorem}

\newtheorem{lemma}[remark]{Lemma}

\theoremstyle{definition}
\newtheorem{definition}[remark]{Definition}

\newcommand{\R}{\mathbb R}

\newcommand{\ind}[1]{{\mathbbm{1}}_{#1}}
\newcommand{\bs}[1]{{\boldsymbol{#1}}}

\begin{document}

\title{The Rescaled P\'olya Urn: local reinforcement and 
    chi-squared goodness of fit test} 
\author{ Giacomo Aletti
\footnote{ADAMSS Center,
  Universit\`a degli Studi di Milano, Milan, Italy, giacomo.aletti@unimi.it}
$\quad$and$\quad$ 
Irene Crimaldi
\footnote{IMT School for Advanced Studies, Lucca, Italy, 
irene.crimaldi@imtlucca.it}
} 
\maketitle

\abstract{Motivated by recent studies of big samples, this work aims at
constructing a parametric model which is characterized by the
following features: (i) a ``local'' reinforcement, i.e.~a
reinforcement mechanism mainly based on the last observations, (ii) a
random persistent fluctuation of the predictive mean, and (iii) a
long-term convergence of the empirical mean to a deterministic limit,
together with a chi-squared goodness of fit result.  This triple
purpose has been achieved by the introduction of a new variant of the
Eggenberger-P\'olya urn, that we call the ``Rescaled'' P\'olya urn. We
provide a complete asymptotic characterization of this model, pointing
out that, for a certain choice of the parameters, it has properties
different from the ones typically exhibited from the other urn models
in the literature. Therefore, beyond the possible statistical
application, this work could be interesting for those who are
concerned with stochastic processes with reinforcement.  \\

\noindent {\em keywords:} empirical mean; central limit theorem; 
chi-squared test; compact Markov chain; P\'olya urn; predictive mean; 
preferential attachment; reinforcement learning; 
reinforced stochastic process; urn model.}


\section{Introduction: framework and motivation}
The well-known Pearson's chi-squared test of goodness of fit is a
statistical test applied to categorical data to establish whether an
observed frequency distribution differs from a theoretical probability
distribution. In this test the observations are always assumed to be
i.i.d., that is independent and identically distributed.  Under this
hypothesis, in a multinomial sample of size $N$, the chi-squared
statistics
\begin{equation}\label{eq:chi1_start}
\chi^2 = 
\sum_{i=1}^k \frac{( O_i - E_i)^2}{E_i} = 
N \sum_{i=1}^k \frac{( \widehat{p}_i - p_i)^2}{p_i}
\end{equation}
(where $k$ is the number of possible values and $O_{i}$, $E_{i}$,
$\widehat{p}_i=O_i/N$ and ${p}_i=E_i/N$ are the observed and expected
absolute and relative frequencies, respectively) is proportional to
$N$, that multiplies the chi-squared distance between the observed and
expected probabilities.  Therefore, the goodness of fit test based on
this statistics is highly sensitive to the sample size $N$ (see, for
instance, \cite{bergh, knoke2002}): the larger $N$, the more
significant a small value of the chi-squared distance. More precisely,
the value of the chi-squared distance has to be compared with the
``critical'' value $\chi^2_{1-\theta}(k-1)/N$, where $
\chi^2_{1-\theta}(k-1)$ denotes the quantile of order $1-\theta$ of
the chi-squared distribution $\chi^2(k-1)$ with $k-1$ degrees of
freedom. Hence, it is clear that the larger $N$, the easier the
rejection of $H_0$. As a consequence, in the context of ``big data''
(e.g.~\cite{bergh, BERTONI2018}), where one often works with
correlated noised data, suitable generative models and related
chi-squared goodness of fit tests are needed.  \\ \indent Different
types of correlation have been taken into account and different
techniques have been developed to control the performance of the
goodness of fit test based on \eqref{eq:chi1_start} (see, among
others, \cite{bergh, chanda, MR0403125, gleser, MR1894384, PeTaGu08,
  radlow, MR3190613}, where some form of correlation is introduced in
the sample and variants of the chi-squared statistics are proposed and
analyzed mainly by means of simulations).  Our approach differs from
the one adopted in the previously quoted papers. Indeed, our starting
point is that a natural way to get a positive correlation between
events of the same type is to deal with the Dirichlet-Multinomial
(D-M) distribution: briefly, the parameters of the D-M distribution is
randomized a priori with a Dirichlet distribution, obtaining an
exchangeable (not independent) sequence.  The variance-covariance
matrix of the D-M distribution is equal to the one of the Multinomial
(M) distribution, multiplied by a fixed constant greater than $1$:
precisely, given the $k$ parameters $\bs{b_0} = ({b_{0\,1}}, \ldots,
{b_{0\,k}})$ of the D-M distribution and setting $|\bs{b_0}| =
\sum_{i=1}^k {b_{0\,i}}$, we have
\begin{align*}
Var_{\text{D-M}} (O_i) & = N \frac{b_{0\,i}}{|\bs{b_0}|} 
\Big( 1 - \frac{b_{0\,i}}{|\bs{b_0}|} \Big)
\frac{N+|\bs{b_0}| }{1+ |\bs{b_0}|} = 
Var_{\text{M}} (O_i)  \frac{N+|\bs{b_0}| }{1+ |\bs{b_0}|} ,
\\
Cov_{\text{D-M}} (O_i,O_j) & = -N \frac{b_{0\,i}b_{0\,j}}{|\bs{b_0}|^2} 
\frac{N+|\bs{b_0}|}{1+ |\bs{b_0}|} = 
Cov_{\text{M}} (O_i,O_j)  \frac{N+|\bs{b_0}| }{1+ |\bs{b_0}|} , 
\qquad \mbox{for } i\neq j.
\end{align*}
Therefore, if we set $|\bs{b_0}| = \frac{1-\rho^2}{\rho^2}$, we have,
for any $i,j \in \{1, \ldots, k\}$
\begin{equation}\label{eq:var_cov-proportional}
Cov_{\text{D-M}} (O_i,O_j) = 
\big(1+ (N-1)\rho^2\big) \, Cov_{\text{D}} (O_i,O_j)  ,
\end{equation}
where $\rho$ represents a correlation parameter.  Roughly speaking,
the Dirichlet-Multinomial model adds variance to the multinomial model
by taking a mixture or by adding a positive correlation.  Property
\eqref{eq:var_cov-proportional} is fundamental for our purpose. In
fact, as well highlighted in \cite{RaoScott81}, the two conditions (i)
$\widehat{p}_i=O_i/N \to p_i$ almost surely for $N\to\infty$ and (ii)
$Cov(O_i,O_j)=\lambda Cov_{\text{D}}(O_i,O_j)$ with $\lambda>1$ imply
that the statistics $\chi^2$, defined in \eqref{eq:chi1_start}, is
asymptotically distributed as $\chi^2(k-1)\lambda$ (see
\cite[Corollary~2]{RaoScott81}), so that the critical value for the
chi-squared distance becomes $ \chi^2_{1-\theta}(k-1)\lambda/N$, where
$\lambda$ mitigate the effect of $N$. As already observed, the D-M
model satisfies (ii), but it is well-known that it does not meet
condition (i). In this paper we give a parametric extension of the D-M
model so that both of the above conditions hold true.  \\ \indent The
Dirichlet-Multinomial distribution may be generated by means of the
standard Eggenberger-P\'olya urn (see \cite{EggPol23, mah}), a model
that has been widely studied and generalized (some recent variants can
be found in \cite{AlGhPa, AlGhRo, AlGhVi, BeCrPrRi16, cal-che-cri-pam,
  caron2017, chen2013, collevecchio2013, Cr16, ghiglietti2014,
  ghiglietti2017, lar-pag}). This urn model with $k$-colors works as
follows.  An urn contains $N_{0\, i}$ balls of color $i$, for
$i=1,\dots, k$, and, at each discrete time, a ball is drawn out from
the urn and then it is put again inside the urn together with
$\alpha>0$ additional balls of the same color. Therefore, if we denote
by $N_{n\, i}$ the number of balls of color $i$ in the urn at time
$n$, we have for $n\geq 1$
\begin{equation*}
N_{n\, i}=N_{n-1\,i}+\alpha\xi_{n\,i},
\end{equation*}
where $\xi_{n\,i}=1$ if the extracted ball at time $n$ is of color
$i$, and $\xi_{n\,i}=0$ otherwise. The parameter $\alpha$ regulates
the reinforcement mechanism: the greater $\alpha$, the greater the
dependence of $N_{n\,i}$ on $\sum_{m=1}^n\xi_{m\,i}$.  In addition, it
is well known that the \emph{conditional expectation} of the
sequential extractions, i.e.~$E[\xi_{n+1\,i} |\, \mbox{``past''}]$,
also known as the~\emph{predictive mean}, converges almost surely to a
beta-distributed random variable, forcing the {\em empirical mean}
$\bar{\xi}_{N\,i} = \sum_{n=1}^N\xi_{n\,i}/N$ to converge almost
surely to the same limit.  \\ \indent In this work we exhibit an urn
model that preserves the relevant aspects of the models above: a
reinforcement mechanism, together with a global almost sure
convergence of the empirical mean of the sequential extraction toward
a fixed limit.  However, differently from the previous models, for a
certain choice of the parameters, the predictive mean $E[\xi_{n+1\,i}
  |\, \mbox{``past''}]$ randomly fluctuates without converging almost
surely, forming asymptotically a stationary ergodic process.  As a
consequence, since the classical martingale approach and the standard
stochastic approximation require or imply the convergence of
$E[\xi_{n+1\,i} |\, \mbox{``past''}]$ (e.g.~\cite{AlCrGh, BeCrPrRi11,
  lar-pag}), in oreder to prove asymptotic results for the introduced
new urn model, we need mathematical methods that are not usual in urn
modeling literature.

\subsubsection*{``Rescaled'' P\'olya urn}
We introduce a new variant of the Eggenberger-P\'olya urn with
$k$-colors, that we call the ``Rescaled'' P\'olya (RP) urn model. In
this model, the almost sure limit of the empirical mean of the draws
will play the r\^ole of an intrinsic long-run characteristic of the
process, while a local mechanism generates persistent fluctuations.
More precisely, the RP urn model is characterized by the introduction
of the parameter $\beta$, together with the initial parameters
$(b_{0\, i})_{i=1,\dots,k}$ and $(B_{0\, i})_{i=1,\dots,k}$, next to
the parameter $\alpha$ of the original model, so that
\begin{equation}\label{eq-dynamics-intro}
\begin{aligned}
N_{n\, i}& =b_{0\, i}+B_{n\, i} &&\text{with }
\\
B_{n\, i}&=\beta B_{n-1\, i}+\alpha\xi_{n\, i}&& n\geq 1.
\end{aligned}
\end{equation}
Therefore, the urn initially contains $b_{0\,i}+B_{0\,i}$ balls of
color $i$ and the parameter $\beta\geq 0$, together with $\alpha>0$,
regulates the reinforcement mechanism. More precisely, $N_{n\,i}$ is
the sum of three terms:
\begin{itemize}
\item the term $b_{0\,i}$, which remains constant along time; 
\item the term $\beta B_{n-1\,i}$, which links $N_{n\,i}$ to the
  ``configuration'' at time $n-1$, through the ``scaling'' parameter
  $\beta$ that tunes the dependence on this factor;
\item the term $\alpha\xi_{n\,i}$, which links $N_{n\,i}$ to the
  outcome of the extraction at time $n$, through the parameter $\alpha$
  that tunes the dependence on this factor.
\end{itemize}
Note that the case $\beta=1$ corresponds to the standard
Eggenberger-P\'olya urn with an initial number
$N_{0\,i}=b_{0\,i}+B_{0\,i}$ of balls of color $i$; while, when
$\beta\neq 1$, the RP urn does not fall in the variants of the
Eggenberger-P\'olya urn discussed in \cite[Section 3.2]{pemantle2007}
and, as explained in details in Section~\ref{urn-model}, it does not
belong to the class of Reinforced Stochastic Processes studied in
\cite{AlCrGh, ale-cri-ghi-WEIGHT-MEAN, ale-cri-ghi-MEAN,
  cri-dai-lou-min, CrDPMi, dai-lou-min, sah}.\\ \indent The quantities
$p_{0\,1},\ldots,p_{0\,k}$ defined as
\begin{equation}\label{def-p_0-intro}   
p_{0\,i}=\frac{b_{0\,i}}{\sum_{i=1}^k b_{0\,i}}
\end{equation}
can be seen as an intrinsic probability distribution on the possible
values (colors) $\{1,\dots,k\}$, that remains constant along time, and
that will be related to the long-term characteristic of the process;
while the random variables $(B_{n\,1},\dots, B_{n\,k})$ model random
fluctuations during time so that the probability distribution on the
set of the $k$ possible values at time $n$ is given by
\begin{equation*}
\psi_{n\,i}=\frac{N_{n\, i}}{\sum_{i=1}^k N_{n\, i}}=
\frac{b_{0\, i}+B_{n\, i}}{\sum_{i=1}^k b_{0\, i}+\sum_{i=1}^k B_{n\, i}}.
\end{equation*}
Assuming for $B_{n\,i}$ the dynamics \eqref{eq-dynamics-intro} with
$\beta>0$, the probability $\psi_{n\,i}$ results increasing with the
number of times we observed the value $i$ (see the following equation
\eqref{eq-psi_n}) and so the random variables $\xi_{n\,i}$ are
generated according to a reinforcement mechanism. But, in particular,
when $\beta<1$, the reinforcement at time $n$ associated to
observation $\xi_{m\,i}$, with $m=1,\dots,n$, increases exponentially
with $m$ (we refer again to the following equation \eqref{eq-psi_n}),
leaving the fluctuations be driven by the most recent draws. We refer
to this feature as ``local'' reinforcement. The case $\beta=0$ is an
extreme case where $\psi_{n\,i}$ depends only on the last draw
$\xi_{n\,i}$ (and not on $\xi_{m\,i}$, with $m=1,\dots,n-1$). Hence,
we are mainly interested in the case $\beta\in [0,1)$, because in this
  case the RP urn exhibits the following distinctive characteristics:
\begin{itemize}
\item[(a)] for each $i$, the process $(\psi_{n\,i})_n$ randomly
  fluctuates, driven by the most recent observations (``local''
  reinforcement), and does not converge almost surely;
\item[(b)] for each $i$, the empirical mean $\bar{\xi}_{N\,i} =
  \sum_{n=1}^N\xi_{n\,i}/N$, that is the empirical frequency $O_i/N$,
  converges almost surely to the deterministic limit $p_i$;
\item[(c)] the chi-squared statistics \eqref{eq:chi1_start} is
  asymptotically distributed as $\chi^2(k-1)\lambda$ with $\lambda>1$.
\end{itemize}
As said before, due to (a), the usual methods adopted in the urn
literature do not work for $\beta < 1$ and so different techniques are
needed for the study of the RP urn model.\\ \indent We have also
considered the asymptotic results for $\beta>1$, to complete the study
of the RP urn model. In this situation, the process $(\psi_{n\,i})_n$
converges exponentially fast to a random limit, and so even faster
than in the classical Eggenberger-P\'olya urn. Therefore, in this
case, we may apply the usual martingale technique (e.g.~\cite{AlCrGh,
  BeCrPrRi11, lar-pag}).  \\
\begin{figure}[tbp]
\begin{center}
\fbox{\includegraphics[width= 0.65\textwidth]{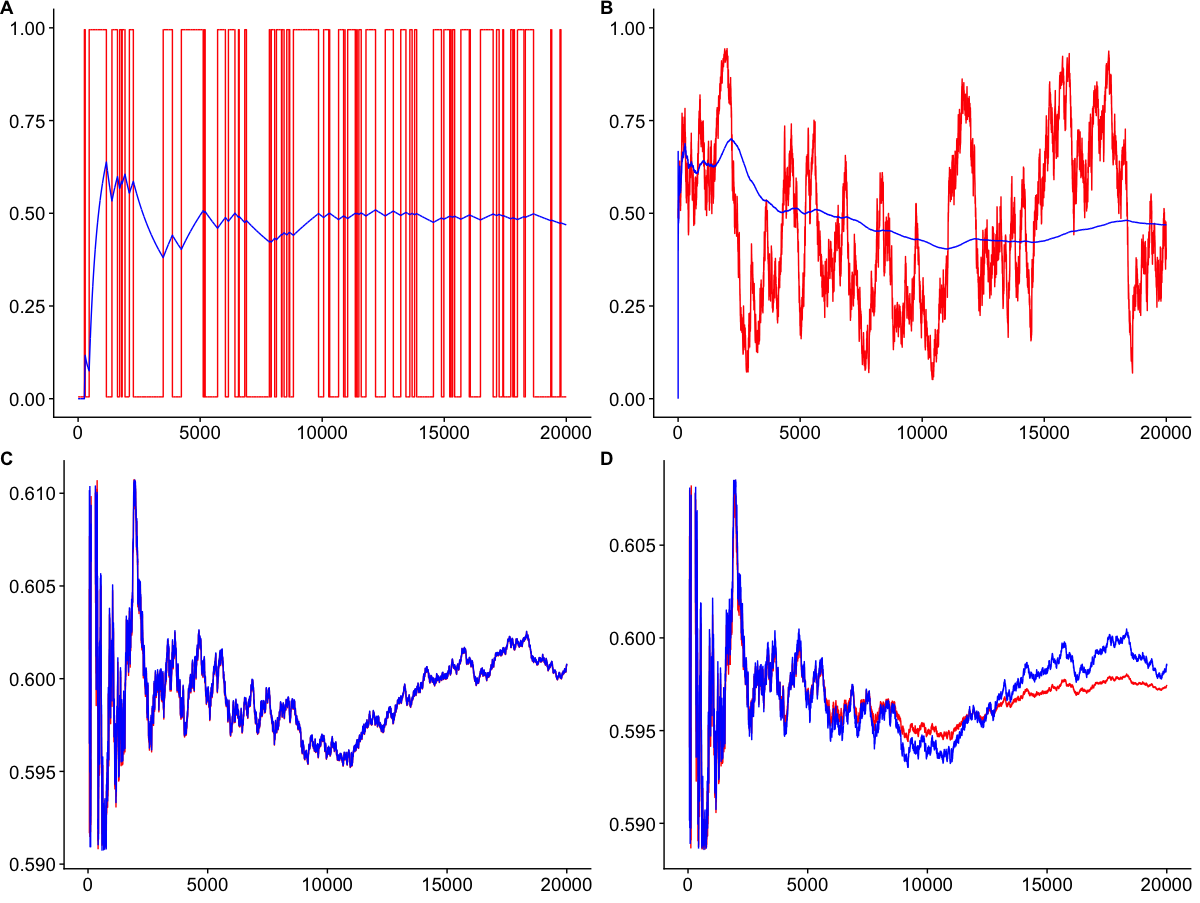}}
\end{center}
\caption{Simulations of the two processes $(\psi_{n\, 1})_n$ (red
  color) and $(\bar{\xi}_{n\,1} )_n$ (blue color), with $n = 1,
  \ldots, 20000$, $p_{0\,1} = \frac{1}{2}$ and for different values of
  $\alpha$ and $\beta$: (A) $\alpha = 199$, $\beta = 0$; (B) $\alpha =
  1$, $\beta = 0.975$; (C) $\alpha = 1$, $\beta = 1$; (D) $\alpha =
  0.5$, $\beta = 1.0001$.  As shown, when $\beta <1$, $(\psi_{n\,
    1})_n$ exhibits a persistent fluctuation, locally reinforced, and
  $(\bar{\xi}_{n\,1} )_n$ converges to the deterministic limit
  $p_{0\,1}$. When $\beta \geq 1$, the $y$-axis is zoomed to show the
  random fluctuations of both the processes towards the same random
  limit.} \label{fig:comparison_Beta}
\end{figure}
\indent In Figure~\ref{fig:comparison_Beta} we show the properties (a)
and (b) for $\beta = 0$ and $\beta \in (0,1)$
(Figure~\ref{fig:comparison_Beta}(A) and
Figure~\ref{fig:comparison_Beta}(B), respectively) compared with the
classical behavior of the processes for $\beta = 1$ and $\beta > 1$
(Figure~\ref{fig:comparison_Beta}(C) and
Figure~\ref{fig:comparison_Beta}(D), respectively).

\subsubsection*{Goodness of fit result}
Given a sample $(\bs{\xi_{1}}, \ldots, \bs{\xi_{N}})$ (where
$\bs{\xi_n}$ denotes the random vector with components $\xi_{n\,i}$,
$i=1,\dots,k$) generated by a RP urn, the statistics
\begin{equation*}
O_i = \#\{n \colon \xi_{n\,i}=1\}=\sum_{n=1}^N \xi_{n\,i}, 
\, \qquad i = 1, \ldots, k,
\end{equation*}
counts the number of times we observed the value $i$.  The theorem
below shows that, when $\beta\in [0,1)$, we can construct a
  chi-squared test for the intrinsic long-run probabilities
  $p_{0\,1},\dots,p_{0\,k}$. More precisely, we will prove the
  following result:

\begin{theorem}\label{th-chi-squared-test}
Assume $p_{0\,i}>0$ for all $i=1,\dots,k$ and $\beta \in
[0,1)$. Define the constants $\gamma$ and $\lambda$ as
\begin{equation}\label{eq:def_gamma}
\gamma = \beta + (1-\beta) \frac{\alpha}
{(1-\beta)\sum_{i=1}^k b_{0\,i} + {\alpha}}
\in (\beta,1)
\qquad\mbox{and }
\end{equation}
\begin{equation}\label{def-lambda}
\lambda = \frac{(1-\beta)^2}{(\gamma-\beta)^2 + (1-\gamma^2)}
\Big( 1 + 2
\frac{\gamma}{1-\gamma}
\Big)>1.
\end{equation}
Then  $O_i/N\stackrel{a.s.}\longrightarrow p_{0\,i}$ and  
\[
\sum_{i=1}^{k} \frac{(O_{i} - N p_{0\,i})^2}{Np_{0\,i}} 
\mathop{\longrightarrow}^{d}_{N\to\infty}
W_{*}= \lambda W_{0}
\]
where $W_{0}$ has distribution
$\chi^2(k-1)=\Gamma\big(\frac{(k-1)}{2},\frac{1}{2})$ and,
consequently, $W_{*}$ has distribution $\Gamma\big(\frac{k-1}{2},
\frac{1}{2\lambda}\big)$.
\end{theorem}

\subsubsection*{Statistical application}
A possible application we have in mind was inspired by
\cite{BERTONI2018, AM2019} and is the following.  We suppose to have a
sample $\{\bs{\xi}_n:\,n=1,\dots,N\}$, where the observations can not
be assumed i.i.d, but they exhibit a structure in clusters, with
independence between clusters and with correlation inside each
cluster. This is a usual circumstance in many applications
(e.g.~\cite{chessa, ieva, tharwat, xu}). More precisely, we consider
the situation when inside each cluster the probability that a certain
unit chooses the value $i$ is affected by the number of units in the
same cluster that have already chosen the value $i$, hence according
to a reinforcement rule. For example, we can imagine that our dataset
collects messages from the on-line social network Twitter: ``tweets''
referring to different topics can be placed in different clusters. If
the topics are distant each other, we can assume independence between
clusters. Inside each cluster, the tweets are temporally ordered and
the associated ``sentiment'' is observed to be driven by a local
reinforcement mechanism: the probability to have a tweet with positive
sentiment is increasing with the number of past tweets with positive
sentiment, but the reinforcement is mostly driven by the most recent
tweets, leading to a fluctuations of the predictive means. A different
clustering of the tweets can be obtained with different slots of time,
sufficiently far from each other.  Another example is the
following. Each cluster corresponds to an agent. The agents act
independently of each other (Independence between clusters). At each
time-step each agent has to choose between $k$ brands, that are
related to a loyalty program: the more he/she selects the same brand,
the more loyalty points he/she gain. This fact induces the
reinforcement mechanism and it could make sense that the reinforcement
is mostly driven by the most recent actions. Finally, we can have the
case where clusters are associated to some products and, at each
time-step a customer has to give a vote to each product on an on-line
platform. Each cluster collects the votes for the corresponding
product. If the products belong to very different categories, we can
assume independence between clusters; while, if the customers can see
the votes given by the previous customers, we can have a reinforcement
mechanism, mainly based on the last observations.  \\

\indent Formally, we suppose that the $N$ units are ordered so that we
have the following $L$ clusters of units:
\begin{equation*}
  C_{\ell}=\left\{\sum_{l=1}^{\ell-1}N_l+1,\dots, \sum_{l=1}^\ell N_l\right\},
  \qquad \ell=1,\dots,L.
\end{equation*}
Therefore, the cardinality of each cluster $C_\ell$ is $N_\ell$. We
assume that the units in different clusters are independent, that is
$$
[\bs{\xi_1},\dots, \bs{\xi_{N_1}}],\,\dots\,,
[\bs{\xi_{\sum_{l=1}^{\ell-1}N_l+1}},\dots, \bs{\xi_{\sum_{l=1}^\ell N_l}}],
\,\dots\,,
[\bs{\xi_{\sum_{l=1}^{L-1}N_l+1}},\dots, \bs{\xi_{N}}]
$$ are $L$ independent multidimensional random variables. Moreover, we
assume that the observations inside each cluster can be modeled as a
RP urn with $\beta\in [0,1)$.  We denote by $
  p_{0\,1}(\ell),\dots,p_{0\,k}(\ell)$ the intrinsic long-run
  probabilities for the cluster $C_\ell$, that we assume strictly
  positive, and we assume the same parameter $\lambda>1$ for each
  cluster (not necessarily the same parameters $\alpha$ and $\beta$)
  so that all the $L$ random variables
$$
Q_\ell=
\sum_{i=1}^{k} \frac{\big(O_{i}(\ell) - N_\ell p_{0\,i}(\ell)\big)^2}
{N_\ell p_{0\,i}(\ell)}, 
\text{ with }
O_i(\ell)= \#\{n\in C_\ell \colon \xi_{n\,i}=1\}, 
$$ 
are asymptotically distributed as $\Gamma\big(\frac{k-1}{2},
\frac{1}{2\lambda}\big)$.  Since $Q_1,\dots,Q_L$ are independent
because they refer to different clusters, when all the cluster sizes
$N_\ell$ are large, we can estimate the parameter $\lambda$ by means
of the (asymptotic) maximum likelihood and obtain
$$
\widehat{\lambda}=\frac{\sum_{\ell=1}^L Q_\ell}{L(k-1)}
\stackrel{d}\sim
\Gamma\Big(\frac{L(k-1)}{2}, \frac{L(k-1)}{2\lambda}\Big).
$$ 
Note that $E[\widehat{\lambda}]=\lambda$, that is the estimator is
unbiased. Moreover, $\widehat{\lambda}/\lambda$ has asymptotic
distribution $\Gamma\big(\frac{L(k-1)}{2}, \frac{L(k-1)}{2}\big)$
(that not depends on $\lambda$) and so it can be used in order to
construct asymptotic confidence intervals for $\lambda$.  Moreover,
given certain (strictly positive) intrinsic probabilities
$p_{0\,1}^*(\ell),\dots,p_{0\,k}^*(\ell)$ for each cluster $C_\ell$,
we can use the above procedure with $p_{0\,i}(\ell)=p_{0\,i}^*(\ell)$
for $i=1,\dots, k$ and $\ell=1,\dots,L$ in order to obtain an estimate
$\widehat{\lambda}^*$ of $\lambda$, and then use the statistics
$Q_\ell$ with $p_{0\,i}(\ell)=p_{0\,i}^*(\ell)$ and the corresponding
asymptotic distribution $\Gamma\big(\frac{k-1}{2},
\frac{1}{2\widehat{\lambda}^*}\big)$ in order to perform a
$\chi^2$-test with null hypothesis
$$
H_0:\quad p_{0\,i}(\ell)=p_{0\,i}^*(\ell)\quad\forall i=1,\dots,k.
$$ 
Regarding the probabilities $p_{0\,i}^*(\ell)$, some possibilities
are:
\begin{itemize}
 \item we can take $p_{0\,i}^*(\ell)=1/k$ for all $i=1,\dots,k$ if we
   want to test possible differences in the probabilities for the $k$
   different values;
 \item we can suppose to have two different periods of times, and so
   two samples, say $\{\bs{\xi^{(1)}_n}:\,n=1,\dots,N\}$ and
   $\{\bs{\xi^{(2)}_n}:\,n=1,\dots,N\}$, take
   $p_{0\,i}^*(\ell)=\sum_{n\in C_\ell}\xi^{(1)}_{n\,i}/N_\ell$ for
   all $i=1,\dots, k$, and perform the test on the second sample in
   order to check possible changes in the intrinsic long-run
   probabilities;
  \item we can take one of the clusters as benchmark, say $\ell^*$,
    set $p_{0\,i}^*(\ell)=\sum_{n\in C_{\ell^*}}\xi_{n\,i}/N_{\ell^*}$
    for all $i=1,\dots, k$ and $\ell\neq\ell^*$, and perform the test
    for the other $L-1$ clusters in order to check differences with
    the benchmark cluster $\ell^*$.
\end{itemize}

\subsubsection*{Structure of the paper}
Summing up, the sequel of the paper is so structured. In Section
\ref{urn-model} we set up our notation and we formally define the RP
urn model with parameters $\alpha>0$ and $\beta\geq 0$. In Section
\ref{study-urn-model} we provide a complete characterization of the RP
urn for the three cases $\beta=0$, $\beta\in [0,1)$ and $\beta>1$. (We
  do not deal with the case $\beta=1$ because, as said before, it
  coincides with the standard Eggenberger-P\'olya urn, whose
  properties are well-known).  In particular, we show that, for each
  $i$, the empirical mean of the $\xi_{n\,i}$ almost surely converges
  to the intrinsic probabilities $p_{0\,i}$ when $\beta\in [0,1)$;
    while it almost surely converges to a random limit when $\beta>
    1$. We obtain also the corresponding CLTs, that, in particular for
    $\beta\in [0,1)$, are the basis for the proof of Theorem
      \ref{th-chi-squared-test}.  For completeness, we also describe
      the case $\alpha = 0$, that generates a sequence of independent
      draws. Section \ref{proof-chi-squared-test} contains the proof
      of Theorem \ref{th-chi-squared-test}, which gives the
      possibility to construct a chi-squared test for the intrinsic
      long-run probabilities when the observed sample is assumed to be
      generated by a RP urn with $\beta\in [0,1)$. Finally, the paper
        contains an Appendix: in Section \ref{sec-clt-markov} we state
        and prove a general CLT for Markov chains with a compact state
        space $S\subset\mathbb{R}^k$, under a certain condition, that
        we call ``linearity'' condition, and in Section
        \ref{sec-coupling} we explain a fundamental coupling technique
        used in the proof of the CLT for $\beta\in (0,1)$.


\section{The ``Rescaled'' P\'olya urn model}
\label{urn-model}

In all the sequel (unless otherwise specified) we suppose given two
parameters $\alpha>0$ and $\beta \geq 0$.  Given a vector $\bs{x}=
(x_1, \ldots, x_k)^\top\in \mathbb{R}^k$, we set \( |\bs{x}| =
\sum_{i=1}^k |x_i| \) and \( \|\bs{x}\|^2 = \bs{x}^\top \bs{x}=
\sum_{i=1}^k |x_i|^2 \).  Moreover we denote by $\bs{1}$ and $\bs{0}$
the vectors with all the components equal to $1$ and equal to $0$,
respectively, and by \(\{\bs{e_1}, \ldots, \bs{e_k}\}\) the canonical
base of $\mathbb{R}^k$.
\\

To formally work with the RP urn model presented in the introduction, we add
here some notations. As in \eqref{eq-dynamics-intro}, the urn
initially contains a constant number of $b_{0\,i}$ distinct balls of
color $i$, with $i=1,\dots,k$, together with a constant number
$B_{0\,i}$ balls of the same color $i$.  We set
$\bs{b_0}=(b_{0\,1},\dots,b_{0\,k})^{\top}$ and
$\bs{B_0}=(B_{0\,1},\dots,B_{0\,k})^{\top}$. In all the sequel (unless
otherwise specified) we assume $|\bs{b_0}|>0$ and
$b_{0\,i}+B_{0\,i}>0$ for each $i=1,\dots, k$.  Consistently with
\eqref{def-p_0-intro}, we set $\bs{p_0} =
\frac{\bs{b_0}}{|\bs{b_0}|}$.  At each discrete time $(n+1)\geq 1$, a
ball is drawn at random from the urn, obtaining the random vector
$\bs{\xi_{n+1}} = (\xi_{n+1\,1}, \ldots, \xi_{n+1\,k})^\top$ defined
as
\begin{equation*}
\xi_{n+1\,i} = 
\begin{cases}
1  &  \text{when the extracted ball at time $n+1$ is of color $i$}
\\
0  & \text{otherwise},
\end{cases}
\end{equation*}
and the number of balls in the urn is so updated:
\begin{equation}\label{eq:reinf1:K}
\bs{N_{n+1}}=\bs{b_0}+\bs{B_{n+1}}\qquad\text{with}
\qquad
\bs{B_{n+1}} = \beta \bs{B_n} + \alpha \bs{\xi_{n+1}}\,,
\end{equation}
which gives (since $|\bs{\xi_{n+1}}|=1$)
\begin{equation}\label{eq:reinf1:K-bis}
|\bs{B_{n+1}}|= \beta |\bs{B_n}| + \alpha.  
\end{equation}
Therefore, setting $r^*_{n} = |\bs{N_n}|= |\bs{b_0}|+|\bs{B_n}|$, we
get
\begin{equation}\label{dinamica-rnstar}
r_{n+1}^*=r_n^*+(\beta-1)|\bs{B_n}|+\alpha.
\end{equation}
Moreover, setting $\mathcal{F}_0$ equal to the trivial $\sigma$-field
and $\mathcal{F}_n=\sigma(\bs{\xi_1},\dots,\bs{\xi_n})$ for $n\geq 1$,
the conditional probabilities $\bs{\psi_{n}}= (\psi_{n\,1}, \ldots,
\psi_{n\,k})^\top$ of the extraction process, also called {\em predictive
means}, are
\begin{equation}\label{eq:extract1a:K-vettoriale}
\bs{\psi_{n}}=E[ \bs{\xi_{n+1}}| \mathcal{F}_n ]= 
\frac{\bs{N_n}}{|\bs{N_n}|} =
\frac{\bs{b_0}+\bs{B_n}}{r_n^*}\qquad\mbox{for } n\geq 0.
\end{equation}
 It is obvious that we have $|\bs{\psi_{n}}|=1$. Finally, for the
 sequel, we set $\bs{\overline{\xi}_{N}}=\sum_{n=1}^N\bs{\xi_{n}}/N$.
 \\

\indent We note that, by means of \eqref{eq:extract1a:K-vettoriale},
together with \eqref{eq:reinf1:K} and \eqref{dinamica-rnstar}, we have
\begin{equation}\label{dinamica-psin}
\bs{\psi_{n}}-\bs{\psi_{n-1}}=
-\frac{(1-\beta)|\bs{b_0}|}{r_{n}^*}\big(\bs{\psi_{n-1}}-\bs{p_0}\big)
+
\frac{\alpha}{r_{n}^*}\big(\bs{\xi_{n}}-\bs{\psi_{n-1}}\big).
\end{equation}
As said before, the RP urn for $\beta=1$ coincides with the well-known
standard Eggenberger-P\'olya urn and so we will exclude it from the
following analyses. When $\beta\neq 1$, since the first term in the
right hand of the above relation, the RP urn does not belong to the
class of Reinforced Stochastic Processes (RSPs) studied in
\cite{AlCrGh, ale-cri-ghi-WEIGHT-MEAN, ale-cri-ghi-MEAN,
  cri-dai-lou-min, CrDPMi, dai-lou-min}. Generally speaking, by
reinforcement in a stochastic dynamics we mean any mechanism for which
the probability that a given event occurs, i.e.~the predictive mean,
has an increasing dependence on the number of times that the same
event occurred in the past. This ``reinforcement mechanism'', also
known as ``preferential attachment rule'' or ``Rich get richer rule''
or ``Matthew effect'', is a key feature governing the dynamics of many
biological, economic and social systems (see,
e.g. \cite{pemantle2007}). The RSPs are characterized by a ``strict''
reinforcement mechanism such that, at each time-step, we have a strict
positive increment of the predictive mean associated to the extracted
color. As an immediate consequence, the ``general'' reinforcement
mechanism is satisfied, that is the predictive mean for a given color
has an increasing dependence on the number of past extractions of that
color. When $\beta\neq 1$, the RP urn model does not satisfy the
``strict'' reinforcement mechanism, because the first term in the
right side of \eqref{dinamica-psin} is positive or negative according
to the sign of $(1-\beta)$ and of
$(\bs{\psi_{n-1}}-\bs{p_0})$. However, when $\alpha,\,\beta>0$, it
satisfies the general reinforcement mechanism. Indeed, by
\eqref{eq:reinf1:K}, \eqref{eq:reinf1:K-bis}, \eqref{dinamica-rnstar}
and \eqref{eq:extract1a:K-vettoriale}, using $\sum_{m=0}^{n-1}
x^m=(1-x^{n})/(1-x)$, we have
\begin{equation}\label{eq-rstar_n}
r_n^*=|\bs{b_0}|+\frac{\alpha }{1-\beta}+
\beta^n\left(|\bs{B_0}|- \frac{\alpha }{1-\beta}\right)
\end{equation}
and 
\begin{equation}\label{eq-psi_n}
\bs{\psi_n} = 
\frac{ 
\bs{b_0}+ \beta^n\bs{B_0} + \alpha \sum_{m=1}^n \beta^{n-m} \bs{\xi_{m}}}
{
|\bs{b_0}|+\frac{\alpha }{1-\beta}+
\beta^n\big(|\bs{B_0}|- \frac{\alpha }{1-\beta}\big)}=
\frac{ 
\beta^{-n}\bs{b_0}+ \bs{B_0} + \alpha \sum_{m=1}^n \beta^{-m} \bs{\xi_{m}}}
{
\beta^{-n}\left(|\bs{b_0}|+\frac{\alpha }{1-\beta}\right)+
|\bs{B_0}|- \frac{\alpha }{1-\beta}}.
\end{equation}
In particular, for $\beta>1$, the dependence of $\bs{\psi_n}$ on
$\bs{\xi_m}$ exponentially decreases with $m$, because of the factor
$\beta^{-m}$. For $\beta<1$ we have the opposite behaviour, that is
the dependence of $\bs{\psi_n}$ on $\bs{\xi_m}$ exponentially
increases with $m$, because of the factor $\beta^{n-m}$, and so the
main contribution is given by the most recent extractions. We refer to
this phenomenon as ``local'' reinforcement. The case $\beta=0$ is an
extreme case, for which $\bs{\psi_n}$ depends only on the last
extraction $\bs{\xi_n}$: at each time-step $n+1\geq 2$ we extract a
ball from an urn with $b_{0\,i}+\alpha$ balls of color $i$, if $i$ is
the color extracted at time $n$, and $b_{0\,j}$ balls for each color
$j\neq i$. This particular case corresponds to a version of the
so-called ``memory-1 senile reinforced random walk'' on a star-shaped
graph introduced in \cite{holmes}, but the study done in that paper
differs from ours.  Finally, we observe that Equation
\eqref{dinamica-psin} recalls the dynamics of a RSP with a ``forcing
input'' (see \cite{AlCrGh, cri-dai-lou-min, sah}), but the main
difference relies on the fact that, for the RP urn, the sequence
$(r_n^*)$ is such that $r_n^*\to r^*>0$, and so $\sum_n
1/r_n^*=+\infty$ and $\sum_n 1/(r_n^*)^2=+\infty$, when $\beta\in
[0,1)$, and such that $\sum_n 1/r_n^*<+\infty$ (and $\sum_n
  1/(r_n^*)^2<+\infty$) when $\beta>1$. These facts lead to a
  different asymptotic behavior of $(\bs{\psi_n})$. Specifically, for
  the RP urn with $\beta\in [0,1)$, the predictive mean $\bs{\psi_n}$
    randomly fluctuates and does not converge almost surely; while,
    for the RP urn with $\beta>1$, the sequence $(\bs{\psi_n})$ almost
    surely converges to a random variable $\bs{\psi_\infty}$ and
    $|\bs{\psi_n}-\bs{\psi_\infty}|=O(\beta^{-n})$. Instead, for the
    RSP with a ``forcing input'', the almost sure convergence of
    $\bs{\psi_n}$ toward the forcing input (which is a constant) holds
    true and the corresponding rate of convergence depends on a model
    parameter $\gamma\in (1/2,1]$ and equals $n^{-\gamma/2}$.


\section{Properties of the ``Rescaled'' P\'olya urn model}
\label{study-urn-model}

We study separately the three cases $\beta=0$, $\beta\in (0,1)$ and
$\beta>1$.

\subsection{The case $\beta=0$}
\label{sec-beta0}

In this case, by \eqref{eq:reinf1:K}, \eqref{eq:reinf1:K-bis} and
\eqref{dinamica-rnstar}, we have for
all $i = 1, \ldots, k$
\begin{equation}\label{eq:extractBeta0:K}
\psi_{0\,i}=\frac{b_{0\,i}+B_{0\,i}}{|\bs{b_0}|+ |\bs{B_0}| }\qquad\mbox{and}\qquad
\psi_{n\,i}= \frac{b_{0\,i}+ \alpha \xi_{n\,i}}{|\bs{b_0}|+ \alpha }
\qquad\mbox{for } n\geq 1.
\end{equation}
We now focus on $\bs{\psi_n}$ for $n\geq 1$. The process
$(\bs{\psi_{n}})_{n\geq 1}$ is a $k$-dimensional Markov chain with a
finite state space $ S=\{\bs{s_1},\ldots, \bs{s_k}\} $, where
$$
\bs{s_i} = \frac{1}{|\bs{b_0}|+ \alpha } 
\big( b_{0\,1} , \ldots, b_{0\,i}+\alpha , \ldots, b_{0\,k}\big)^\top, 
\qquad \mbox{for } i=1,\ldots, k\,,
$$
and transition probability matrix
\[
P = \frac{1}{|\bs{b_0}|+ \alpha } 
(\bs{1}_k \, \bs{b_0}^\top  + \alpha \mathrm{Id}_k )= 
\frac{|\bs{b_0}|}{|\bs{b_0}|+ \alpha } 
(\bs{1}_k \, \bs{p_0}^\top  + \tfrac{\alpha}{|\bs{b_0}|} \mathrm{Id}_k )\,, 
\]
which is irreducible and aperiodic. Now, since $\bs{1}_k \,
\bs{p_0}^\top $ is idempotent and commutes with the identity, then we
have
\begin{equation}\label{eq:P^n:K}
\begin{aligned}
P^n  & = \Big(\frac{|\bs{b_0}|}{|\bs{b_0}|+ \alpha }\Big)^n  
\bigg( \Big(\sum_{j=0}^{n-1} \binom{n}{j} 
(\tfrac{\alpha}{|\bs{b_0}|})^j \Big)\bs{1}_k \, \bs{p_0}^\top  + 
(\tfrac{\alpha}{|\bs{b_0}|})^n \mathrm{Id}_k \bigg)
\\
& = \Big(\frac{|\bs{b_0}|}{|\bs{b_0}|+ \alpha }\Big)^n  
\bigg( 
\Big( (1+\tfrac{\alpha}{|\bs{b_0}|})^n -(\tfrac{\alpha}{|\bs{b_0}|})^n 
\Big)\bs{1}_k \, \bs{p_0}^\top + 
(\tfrac{\alpha}{|\bs{b_0}|})^n \mathrm{Id}_k \bigg)
\\
& = \bs{1}_k \, \bs{p_0}^\top  +
\Big(\frac{\alpha}{|\bs{b_0}|+ \alpha }\Big)^n  
\Big(\mathrm{Id}_k- \bs{1}_k \, \bs{p_0}^\top \Big) 
\\
& = \bs{1}_k \, \bs{p_0}^\top  +
\gamma^n  
\Big(\mathrm{Id}_k- \bs{1}_k \, \bs{p_0}^\top \Big) 
\,,
\end{aligned}
\end{equation}
where $\gamma$ is the constant given in \eqref{eq:def_gamma}, that
becomes equal to $\frac{\alpha}{|\bs{b_0}| + {\alpha}}$ for $\beta=0$.
We note that $\gamma<1$ (since $|\bs{b_0}|>0$ by assumption) and so
$P^n \to \bs{1}_k \, \bs{p_0}^\top $, and the unique invariant
probability measure on $S$ is hence $ \bs{\pi} = \bs{p_0}$.

\begin{theorem}\label{thm:CLT_Beta0:K}
We have $\bs{\overline{\xi}_N}\stackrel{a.s.}\longrightarrow \bs{p_0}$
and
\[
\sqrt{N}\left(\bs{\overline{\xi}_N}-\bs{p_0}\right)=
\frac{\sum_{n=1}^N (\bs{\xi_n} - \bs{p_0})}{\sqrt{N}}  
\mathop{\longrightarrow}^{d}_{N\to\infty} 
\mathcal{N} (0,\Sigma^2),
\]
where
\begin{equation}\label{eq:sigma2pi_psi_beta0:K}
\Sigma^2 
= \lambda 
\Big(\mathrm{diag}(\bs{p_0})- \bs{p_0} \bs{p_0}^\top \Big), 
\end{equation}
with $\lambda$ defined in \eqref{def-lambda} (taking $\beta=0$).
\end{theorem}
\begin{proof}
We observe that, by \eqref{eq:extractBeta0:K}, we have for each $n\geq 0$
\[
\{\xi_{n+1\,i}=1\} = \{\bs{\psi_{n+1}}=\bs{s_i}\}.
\] 
Therefore, the strong law of large numbers for Markov chains
immediately yields
\[
\bs{\overline{\xi}_N} = 
\frac{1}{N} \sum_{n=1}^N \Big( \ind{\{\bs{\psi_n} = \bs{s_1}\}}, \ldots, 
\ind{\{\bs{\psi_n} = \bs{s_k}\}} \Big)^\top
\stackrel{a.s.}\longrightarrow \bs{p_0}.
\]
Take a vector $\bs{c} = (c_1,\ldots,c_k)^T$ and define $g(\bs{x}) =
\bs{c}^T (\bs{x}-\bs{p_0})$. Recall that
$g(\bs{\xi}_n)=g\Big(\ind{\{\bs{\psi_n} = \bs{s_1}\}}, \ldots,
\ind{\{\bs{\psi_n} = \bs{s_k}\}} \Big)$ and apply the central limit
theorem for uniformly ergodic Markov chains (see, for instance, 
\cite[Theorem~17.0.1]{Meyn09}): 
the sequence $( \frac{\sum_{n=1}^N g (\bs{\xi_n})}{\sqrt{N}})$
converges in distribution to the Gaussian distribution $\mathcal{N}
(0,\sigma^2_{\bs{c}})$, with 
\begin{equation*}
\begin{split}
\sigma^2_{\bs{c}} &= \mathrm{Var}[g (\bs{{\xi}^{(\pi)}_0}) ] + 
2 \sum_{n\geq 1} 
\mathrm{Cov}(g (\bs{{\xi}^{(\pi)}_0}) , 
g (\bs{{\xi}^{(\pi)}_n}))
\\ 
&= \bs{c}^T \Big( \mathrm{Var}[\bs{{\xi}^{(\pi)}_0}] + 
2 \sum_{n\geq 1} \mathrm{Cov}(\bs{{\xi}^{(\pi)}_0} , 
\bs{{\xi}^{(\pi)}_n} ) \Big) \bs{c}\,,
\end{split}
\end{equation*}
where $\bs{{\xi}^{(\pi)}_n}=\Big(\ind{\{\bs{{\psi}^{(\pi)}_n}
  =\bs{s_1}\}}, \ldots, \ind{\{\bs{{\psi}^{(\pi)}_n} = \bs{s_k}\}}
\Big)$ and $\big(\bs{{\psi}^{(\pi)}_n}\big)_{n\geq 0}$ is a Markov
chain with transition matrix $P$ and initial distribution $\bs{\pi}$,
that is $\bs{p_0}$. Now, by definition, $\bs{{\xi}^{(\pi)}_0}
\bs{{\xi}^{(\pi)}_0}^T = \mathrm{diag}(\bs{{\xi}^{(\pi)}_0})$, and
hence $\mathrm{Var}[\bs{{\xi}^{(\pi)}_0} ] = \mathrm{diag}(\bs{p_0}) -
\bs{p_0}\bs{p_0}^T$.  Moreover, by means of \eqref{eq:P^n:K},
\[
E[\bs{{\xi}^{(\pi)}_0} \bs{{\xi}^{(\pi)}_n}^\top] =
\mathrm{diag}(\bs{p_0}) P^n = \bs{p_0} \bs{p_0}^\top  +
\gamma^n  
\Big(\mathrm{diag}(\bs{p_0})- \bs{p_0} \bs{p_0}^\top \Big) .
\]
Hence, since $\gamma < 1$, we have $\sum_{n\geq
  1}\gamma^n=\gamma/(1-\gamma)$ and so
\begin{align*}
\mathrm{Var}[\bs{{\xi}^{(\pi)}_0} ] + 
2 \sum_{n\geq 1} \mathrm{Cov}(  \bs{{\xi}^{(\pi)}_0} , 
\bs{{\xi}^{(\pi)}_n} ) & = 
\Big(\mathrm{diag}(\bs{p_0})- \bs{p_0} \bs{p_0}^\top \Big) 
\Big(1+ \frac{2\gamma}{1-\gamma } \Big).
\end{align*}
By the Cram\'er-Wold device, the theorem is proved with $\Sigma^2$
given in \eqref{eq:sigma2pi_psi_beta0:K}.
\end{proof}

\begin{remark}\label{rem:b0_beta=0}
Note that in Theorem~\ref{thm:CLT_Beta0:K} we do not assume
$b_{0\,i}>0$ for all $i$, but only $|\bs{b_0}|>0$ (as said in
Sec.~\ref{urn-model}).  A different behavior is observed when $
\bs{b_0} = \bs{0}$. In this case, \eqref{eq:extractBeta0:K} gives $
\bs{\psi_n} = \bs{\xi_n} $ for $n\geq 1$. Since
$\psi_{n\,i}=P(\xi_{n\,i}=1|\mathcal{F}_n)$, the above equality
implies recursively $\bs{\psi_{n}} = \bs{\xi_{n}} = \bs{\xi_{1}}$ for
each $n\geq 1$.  In other words, the process of extractions
$\bs{\xi}=(\bs{\xi_n})_{n\geq 1}$ is constant, with $P(\xi_{1\,i} = 1)
= \psi_{0\,i} = {B_{0\,i}} / |\bs{B_0}|$.
\end{remark}


\subsection{The case $\beta \in (0,1)$}
\label{sec-beta-minore1}

In this case, we have $\lim_n\beta^n=0$ and $\sum_{n\geq
  1}\beta^n=\beta/(1-\beta)$. Therefore, setting $r
=\tfrac{\alpha}{1-\beta}$ and $r^*=|\bs{b_0}|+r$, we have by
\eqref{eq:reinf1:K-bis} and \eqref{dinamica-rnstar}
\begin{equation}\label{eq:numBallsTot1:k}
r_n^*=|\bs{b_0}|+|\bs{B_{n}}| 
= 
|\bs{b_0}|+r + \beta^n ( |\bs{B_{0}}| - r) 
\longrightarrow r^*>0,
\end{equation}
and so we have that the denominator $r^*_n$ in $\bs{\psi_n}$ (see
Eq.~\eqref{eq:extract1a:K-vettoriale}) goes exponentially fast to the
limit $r^*$.  Moreover, recalling the definition of the constant
$\gamma$ in \eqref{eq:def_gamma}, we have $\beta < \gamma < 1$
(remember that $|\bs{b_0}|>0$ by assumption) and
\begin{equation}\label{relazioni-gamma}
\gamma-\beta=\frac{\alpha}{r^*}\qquad\mbox{and}\qquad
1-\gamma=\frac{(1-\beta)|\bs{b_0}|}{r^*}.
\end{equation}
Therefore, by \eqref{eq:numBallsTot1:k}, the terms
$\frac{(1-\beta)|\bs{b_0}|}{r_{n}^*}$ and $\frac{\alpha}{r_{n}^*}$ in
the dynamics \eqref{dinamica-psin} converge exponentially fast to
$(1-\gamma)$ and $(\gamma-\beta)$, respectively.  Furthermore, as we
will see, the fact that the constant $\gamma$ is strictly smaller than 
$1$ will play a central r\^ole, because it will imply the existence of
a contraction of the process $\bs{\psi}=(\bs{\psi_{n}})_n$ in a proper
metric space with (sharp) constant $\gamma$.  Consequently, it is not
a surprise that this constant enters naturally in the parameters of
the asymptotic distribution, given in the following result:

\begin{theorem}\label{thm:CLT_beta<1}
We have $\bs{\overline{\xi}_N} \mathop{\longrightarrow}\limits^{a.s.}
\bs{p_0}$ and
\[
{\sqrt{N}} (\bs{\overline{\xi}_N} - \bs{p_0}) 
=
\frac{\sum_{n=1}^N (\bs{\xi_n} - \bs{p_0})}{\sqrt{N}}
\mathop{\longrightarrow}^{d}
\mathcal{N} (\bs{0},\Sigma^2),
\]
where
\[
\Sigma^2 = \lambda
\Big( \mathrm{diag} (\bs{p_0}) - \bs{p_0}\bs{p_0}^T\Big),
\]
with $\lambda$ defined in \eqref{def-lambda} as a function of $\beta$
and $\gamma$.
\end{theorem}

\begin{remark}\label{rem:b0_beta<1}
Note that in Theorem~\ref{thm:CLT_Beta0:K} we do not assume
$b_{0\,i}>0$ for all $i$, but only $|\bs{b_0}|>0$ (as said in
Sec.~\ref{urn-model}). Again, a different behavior is observed when $
\bs{b_0} = \bs{0}$.  Indeed, from \eqref{dinamica-psin}, we have
\begin{equation}\label{dyn-0}
\bs{\psi_{n}}-\bs{\psi_{n-1}}=
\frac{\alpha}{r_{n}^*}\big(\bs{\xi_{n}}-\bs{\psi_{n-1}}\big),
\end{equation}
and hence $(\bs{\psi_n})$ is a martingale.  The asymptotic
result given above fails (in fact, we have $\gamma = 1$).  The
martingale property of the bounded process $\bs{\psi}$ implies that
$\bs{\psi_n}$ converges almost surely (and in mean) to a bounded
random variable $\bs{\psi_\infty}$.  In addition, since \( r_{n}^* \to
\alpha/(1-\beta)\), from \eqref{dyn-0}, we obtain that the unique
possible limits $\bs{\psi_{\infty}}$ are those for which $\bs{\xi_{n}}
= \bs{\psi_{\infty}}$ eventually.  Hence $\bs{\psi_\infty}$ takes
values in $\{\bs{e_1},\dots,\bs{e_k}\}$ and, since we have
$E[\bs{\psi_{\infty}}] = E[\bs{\psi_{0}}] = \bs{B_0}/{|\bs{B_0}|}$, we
get $P( \bs{\psi_{\infty}} = \bs{e_i} ) = \frac{{B_{0\,i}}}{
  |\bs{B_0}|}$ for all $i=1,\dots,k$.
\medskip
\end{remark}

We will split the proof of Theorem~\ref{thm:CLT_beta<1} into two main
steps: first, we will prove that the convergence behaviour of
$\bs{\overline{\xi}_N}$ does not depend on the initial constant
$|\bs{B}_0|$ and, then, without loss of generality, we will assume
$|\bs{B}_0|=r$ and we will give the proof of the theorem under this
assumption.


\subsubsection{Independence of the asymptotic properties 
of the empirical mean from $|\bs{B_0}|$}

We use a coupling method to prove that the convergence results stated
in Theorem~\ref{thm:CLT_beta<1} are not affected by the value of the
initial constant $|\bs{B_0}|$.  \\

\indent Set $\bs{\xi^{(1)}_n}=\bs{\xi_{n}}$ and $\bs{\psi^{(1)}} =
\bs{\psi}$, that follows the dynamics \eqref{dinamica-psin}, together
with \eqref{eq:extract1a:K-vettoriale}, starting from a certain
initial point $\bs{\psi_0^{(1)}}=\bs{\psi_0}$. By \eqref{eq:reinf1:K}
and relations \eqref{relazioni-gamma}, we can write
\begin{equation*}
\begin{aligned}
\bs{\psi^{(1)}_{n+1}} 
&= 
\frac{\bs{b_0} + \beta \bs{B_n} + \alpha \bs{\xi^{(1)}_{n+1}} }{r^*_{n+1}} 
= 
\frac{\bs{b_0} + \beta ( r^*_n\bs{\psi^{(1)}_{n}} - \bs{b_0} ) + 
\alpha \bs{\xi^{(1)}_{n+1}} }{r^*_{n+1}} 
\\
&=\beta \bs{\psi^{(1)}_n}\frac{r^*_n}{r^*_{n+1}}
+\frac{(1-\beta)\bs{b_0}}{r^*_{n+1}}+\frac{\alpha}{r^*_{n+1}}\bs{\xi^{(1)}_{n+1}}
\\
&= \beta \bs{\psi^{(1)}_{n}} + (\gamma - \beta) \bs{\xi^{(1)}_{n+1}} 
+ \bs{l_{n+1}^{(1)}}(\bs{\psi^{(1)}_n},\bs{\xi^{(1)}_{n+1}})+
(1-\gamma)\bs{p_0} , 
\end{aligned}
\end{equation*}
where 
$$
\bs{l_{n+1}^{(1)}}(\bs{x},\bs{y})=
\Big(\frac{r^*}{r^*_{n+1}}-1\Big) \big[ (1-\gamma)\bs{p_0}
+ (\gamma-\beta) \bs{y} \big]
+\Big(\frac{r^*_n}{r^*_{n+1}}-1\Big) \beta\bs{x}.
$$ Since, by \eqref{eq:numBallsTot1:k}, we have
$r^*/r_{n+1}^*-1=O(\beta^{n+1})$ and $r_{n}^*/r_{n+1}^* -1 =
O(\beta^{n+1})$, we get $|\bs{l_{n+1}^{(1)}}|=O(\beta^{n+1})$. Now,
take $\bs{\xi^{(2)}}=(\bs{\xi^{(2)}_n})_n$ and $\bs{\psi^{(2)}} =
(\bs{\psi^{(2)}_{n}} )_n$ following the same dynamics given in
\eqref{eq:extract1a:K-vettoriale} and \eqref{dinamica-psin}, but
starting from an initial point with $|\bs{B_0^{(2)}}|=r$. Therefore,
we have
\begin{equation*}
\bs{\psi^{(2)}_{n+1}} 
=\beta\bs{\psi^{(2)}_{n}}+(\gamma - \beta)\bs{\xi^{(2)}_{n+1}}+(1-\gamma)\bs{p_0}.
\end{equation*}
Both dynamics are of the form \eqref{eq:evolution_psi:K_modified} with
$a_0 = \beta$, $a_1 = (\gamma-\beta)$, $\bs{c}= (1-\gamma)\bs{p_0}$,
$c_n^{(1)}=\beta^{n}$ and $c_{n}^{(2)}=0$, and, by
\eqref{eq:extract1a:K-vettoriale}, condition \eqref{eq:cond-prob-g}
holds true. Hence we can apply Theorem~\ref{th:coupl_processes} so
that there exist two stochastic processes
$\bs{\widetilde{\psi}^{(1)}}$ and $\bs{\widetilde{\psi}^{(2)}}$,
following the dynamics \eqref{eq:evolution_psi:K_modified-bis} (with
the same specifications as above), together with
\eqref{eq:cond-prob-g-bis}, starting from the same initial points and
such that \eqref{dis-2} holds true, that is
\[
\begin{aligned}
E\Big[ |\bs{\widetilde{\psi}^{(1)}_{n+1}}-\bs{\widetilde{\psi}^{(2)}_{n+1}} | 
\Big| \bs{\widetilde{\psi}^{(2)}_{0}}, \bs{\widetilde{\psi}^{(1)}_{0}}\Big]
& \leq \gamma^{n+1} 
|\bs{\widetilde{\psi}^{(1)}_{0}}-\bs{\widetilde{\psi}^{(2)}_{0}}|
 + O\Big( \sum_{j=1}^{n+1} \gamma^{n+1-j} \beta^{j}\Big)
\\
& =O\left((n+2) \max( \gamma  , \beta )^{n+1}\right)=
O\big((n+2)\gamma^{n+1}\big) .
\end{aligned}
\]
Since $\gamma<1$, if we subtract
\eqref{eq:evolution_psi:K_modified-bis} with $\ell=2$ by
\eqref{eq:evolution_psi:K_modified-bis} with $\ell=1$, we obtain that
\[
\sum_{n=1}^{+\infty} E \Big[ |\bs{\widetilde{\xi}^{(1)}_{n}} - 
\bs{\widetilde{\xi}^{(2)}_{n}} |
\Big] < +\infty ,
\]
which implies $\sum_{n=0}^{+\infty}|\bs{\widetilde{\xi}^{(1)}_{n}} -
\bs{\widetilde{\xi}^{(2)}_{n}} | <+\infty$ a.s., that is
\begin{equation}\label{eventually=}
\bs{\widetilde{\xi}^{(1)}_{n}} =
\bs{\widetilde{\xi}^{(2)}_{n}}\quad\mbox{ eventually}.
\end{equation}  
Therefore, if we prove some asymptotic results for $\bs{\xi^{(2)}}$,
then they hold true also for $\bs{\widetilde{\xi}^{(2)}}$ (since they
have the same joint distribution), then they hold true also for
$\bs{\widetilde{\xi}^{(1)}}$ (since \eqref{eventually=}), and finally
they hold true also for $\bs{{\xi}^{(1)}}$ (since they have the same
joint distribution). Summing up, without loss of generality, we may
prove Theorem~\ref{thm:CLT_beta<1} under the additional assumption
$|\bs{B_0}| = r$.


\subsubsection{The case $|\bs{B_0}|=r$}
\label{sec:beta<1}
Thanks to what we have observed in the previous subsection, 
we here assume that $|\bs{B_0}|=r=\alpha/(1-\beta)$, that implies 
$|\bs{B_n}|=r$ and 
\begin{equation}\label{num}
r_n^*=r^*=|\bs{b_0}|+r
\end{equation}
for any $n$. Hence, we can
simplify \eqref{eq:extract1a:K-vettoriale} as 
\begin{equation*}
\psi_{n\,i} = P( \xi_{n+1\,i} = 1| \mathcal{F}_n ) = 
\frac{b_{0\,i} + B_{n\,i}}{|\bs{b_0}|+ r}=\frac{b_{0\,i} + B_{n\,i}}{r^*}.
\end{equation*}
The process $\bs{\psi}=(\bs{\psi_{n}})_{n\geq 0}$ is then a
Markov chain with state space
\[
S = 
\left\{\bs{x} \colon
x_i\in 
\Big[\frac{b_{0\,i}}{r^*},\frac{b_{0\,i}+r}{r^*}\Big] ,
 |\bs{x}|= 1\right\},
\]
which, endowed with the distance induced by the norm $|\cdot|$, is a
compact metric space. 
\\

In the sequel, according to the context, since we work with a Markov
chain with state space $S\subset\mathbb{R}^k$, the notation $P$ will
be used for:
\begin{itemize}
\item a kernel $P: S\times \mathcal{B}(S)\to [0,1]$, where
  $\mathcal{B}(S)$ is the Borel $\sigma$-field on $S$ and we will use
  the notation $P(\bs{x},\bs{dy})$ in the integrals;
\item an operator $P:C(S) \to M(S)$, where $C(S)$ and $M(S)$ denote
  the space of the continuous and measurable functions on $S$,
  respectively, defined as
\begin{equation*}
(Pf)(\bs{x}) = \int_S f(\bs{y}) P(\bs{x},\bs{dy}).
\end{equation*} 
In addition, when $f$ is the identity map, that is
$f(\bs{y})=\bs{y}$, we will write $(Pid)(\bs{x})$ or
$(P\bs{y})(\bs{x})$.
\end{itemize}
Moreover, we set $P^0f = f$ and $P^nf = P(P^{n-1}f)$. 
\\

By \eqref{dinamica-psin}, together with \eqref{relazioni-gamma} and
\eqref{num}, the process $\bs{\psi}=(\bs{\psi_n})_n$ follows the
dynamics
\begin{equation}\label{eq:evolution_psi:K}
\bs{\psi_{n+1}} 
= \beta \bs{\psi_{n}} +\bs{b_0}\frac{(1-\beta)}{r^*} + 
\bs{\xi_{n+1}}\frac{\alpha}{r^*}
=\beta \bs{\psi_{n}} +(1-\gamma)\bs{p_0}+ 
(\gamma-\beta)\bs{\xi_{n+1}}. 
\end{equation}
Therefore, given $\bs{z} = (z_1, \ldots, z_k)^T$ and setting
\[
\bs{z}_{(i)} = \Big(z_1, \ldots , z_i+\frac{\alpha}{r^*} , \ldots, z_k\Big)^T
=\Big(z_1, \ldots , z_i+(\gamma-\beta) , \ldots, z_k\Big)^T,
\]
for any $i=1, \ldots, k$, we get 
\begin{equation}\label{eq:operator_P:K}
(Pf)(\bs{x}) = 
E[ f(\bs{\psi_{n+1}}) | \bs{\psi_n}=\bs{x} ] = 
\sum_{i=1}^k x_i f\Big( \big(\beta \bs{x} 
+\bs{p_0}(1-\gamma) \big)_{(i)}\Big). 
\end{equation}
In particular, from the above equality, we get 
\begin{equation}\label{eq-operator-Pid}
(Pid)(\bs{x})-\bs{p_0} = 
E[ \bs{\psi_{n+1}}-\bs{p_0} | \bs{\psi_n}=\bs{x} ] = 
\gamma( \bs{x}-\bs{p_0} ). 
\end{equation}
We now show that $\bs{\psi}$ is an irreducible, aperiodic, compact
Markov chain (see Def.~\ref{def:compact-Markov-process} and
Def.~\ref{def-irreducible}).
\\

\noindent{\bf Check that $\bs{\psi}$ is a compact Markov chain:} By
Lemma~\ref{lem:cMP}, it is sufficient to show that $P$ defined in
\eqref{eq:operator_P:K} is weak Feller (Definition~\ref{weak-Feller})
and that it is a semi-contractive operator on $Lip(S)$
(Definition~\ref{def-star}). From \eqref{eq:operator_P:K}, we have
immediately that the function $Pf$ is continuous whenever $f$ is
continuous and hence $P$ is weak-Feller. In order to prove the
contractive property, we start by observing that the dynamics
\eqref{eq:evolution_psi:K} of $\bs{\psi}$ is of the form
\eqref{eq:evolution_psi:K_modified} with $a_0=\beta$,
$a_1=(\gamma-\beta)$, $\bs{c}=\bs{p_0}(1-\gamma)$ and
$\bs{l_{n}}\equiv\bs{0}$ for each $n$. Moreover, by
\eqref{eq:extract1a:K-vettoriale}, condition \eqref{eq:cond-prob-g}
holds true. Then , let $\bs{\psi^{(1)}}$ and $\bs{\psi^{(2)}}$ be two
stochastic processes following the dynamics
\eqref{eq:evolution_psi:K_modified} with the same specifications as
above, together with \eqref{eq:cond-prob-g}, and starting,
respectively, from the point $\bs{x}$ and $\bs{y}$.  Then, applying
Theorem~\ref{th:coupl_processes}, we get two stochastic processes
$\bs{\widetilde\psi^{(1)}}$ and $\bs{\widetilde\psi^{(2)}}$, evolving
according to \eqref{eq:evolution_psi:K_modified-bis}, together with
\eqref{eq:cond-prob-g-bis}, starting from the same initial points
$\bs{x}$ and $\bs{y}$ and such that
 \begin{equation*}
E\Big[\, | \bs{\widetilde{\psi}_{1}^{(2)}} - \bs{\widetilde{\psi}_{1}^{(1)}}|
\,\big]= 
E\Big[\, | \bs{\widetilde{\psi}_{1}^{(2)}} - \bs{\widetilde{\psi}_{1}^{(1)}} | 
\Big| \bs{\widetilde{\psi}_{0}^{(1)}}= \bs{x} ,\, \bs{\widetilde{\psi}_{0}^{(2)}} 
= \bs{y}  \,\Big] 
\leq \gamma | \bs{x}-\bs{y} |.
\end{equation*}
Therefore, if we take $f \in Lip(S)$ with
\[
|f|_{Lip} = \sup_{\bs{x},\bs{y}\in S,\,\bs{x}\neq \bs{y}} 
\frac{|f(\bs{y})-f(\bs{x})|}{|\bs{y}-\bs{x}|},
\]
we obtain
\begin{align*}
| (Pf)(\bs{y})-(Pf)(\bs{x}) | &=
|
E[f(\bs{\widetilde{\psi}_{1}^{(2)}})|\bs{\widetilde{\psi}_{0}^{(2)}}=\bs{y}] 
- 
E[f(\bs{\widetilde{\psi}_{1}^{(1)}})|\bs{\widetilde{\psi}_{0}^{(1)}}=\bs{x}] 
|
\\
&=
E\big[\, 
|f(\bs{\widetilde{\psi}_{1}^{(2)}}) - f(\bs{\widetilde{\psi}_{1}^{(1)}})|  
\,\big]  
\\ &  \leq 
|f|_{Lip} E\Big[ 
| \bs{\widetilde{\psi}_{1}^{(2)}} - \bs{\widetilde{\psi}_{1}^{(1)}} |
\Big ]  
\leq |f|_{Lip} \ \gamma \ |\bs{x}-\bs{y}|
\end{align*}
and so
\[
|(Pf)|_{Lip} = 
\sup_{\bs{x},\bs{y}\in S,\,\bs{x}\neq \bs{y}} 
\frac{|(Pf)(\bs{y})-(Pf)(\bs{x})|}{|\bs{y}-\bs{x}|} 
\leq \gamma |f|_{Lip},
\]
with $\gamma <1 $, as requested. 
\qed

\medskip

\noindent{\bf Check that $\bs{\psi}$ is irreducible and aperiodic:} We
prove the irreducibility and aperiodicity condition stated in
Def.~\ref{def-irreducible}, using Theorem \ref{th-irreducible}.
Therefore, let us denote by $\pi$ an invariant probability measure for
$P$. Moreover, let $\bs{\psi^{(1)}}$ and $\bs{\psi^{(2)}}$ be two
processes that follows the same dynamics \eqref{eq:evolution_psi:K} of
$\bs{\psi}$, but for the first process, we set the initial
distribution equal to $\pi$, while for the second process, we take any
other initial distribution $\nu$ on $S$.  Again, as above, since
\eqref{eq:evolution_psi:K} is of the form
\eqref{eq:evolution_psi:K_modified} with $a_0=\beta$,
$a_1=(\gamma-\beta)$, $\bs{c}=\bs{p_0}(1-\gamma)$ and
$\bs{l_{n}}\equiv\bs{0}$ for each $n$, and, by
\eqref{eq:extract1a:K-vettoriale}, condition \eqref{eq:cond-prob-g}
holds true, then we can apply Theorem~\ref{th:coupl_processes} and
obtain two stochastic processes $\bs{\widetilde\psi^{(1)}}$ and
$\bs{\widetilde\psi^{(2)}}$, evolving according to
\eqref{eq:evolution_psi:K_modified-bis} (with the same specifications
as above), together with \eqref{eq:cond-prob-g-bis}, starting from the
same initial random variables $\bs{\psi^{(1)}_0}$ (with distribution
$\pi$) and $\bs{\psi^{(2)}_0}$ (with distribution $\nu$) and such that
\[
E\Big[ |\bs{\widetilde{\psi}^{(1)}_{n+1}}-\bs{\widetilde{\psi}^{(2)}_{n+1}} | 
\Big| \bs{\widetilde{\psi}^{(2)}_{0}}, \bs{\widetilde{\psi}^{(1)}_{0}}\Big]
\leq \gamma^{n+1} 
|\bs{\widetilde{\psi}^{(1)}_{0}}-\bs{\widetilde{\psi}^{(2)}_{0}}|.
\]
Hence, since $\gamma<1$, we have
$E[|\bs{\widetilde{\psi}^{(1)}_{n}}-\bs{\widetilde{\psi}^{(2)}_{n}}|]
\longrightarrow 0 $, and, since the distribution of
$\bs{\widetilde{\psi}^{(1)}_{n}}$ is always $\pi$ (by definition of
invariant probability measure), we can conclude that
$\bs{\widetilde\psi^{(2)}_n}$, and so $\bs{\psi^{(2)}_n}$ (because
they have the same distribution), converges in distribution to $\pi$.
\qed \\

\noindent{\bf Proof of the almost sure convergence:} We have already
proven that the Markov chain $\bs{\psi}$ has one invariant probability
measure $\pi$.  Furthermore, from \eqref{eq:evolution_psi:K} we
get
\begin{multline}  \label{eq:f2:K}
\bs{\xi_{n+1}} - \bs{p_0}  
= \frac{1}{\gamma-\beta} 
\big[ \bs{\psi_{n+1}} - \beta \bs{\psi_{n}} - \bs{p_0}(1-\gamma) \big] - 
\bs{p_0}  
\\
= \frac{1}{\gamma-\beta} 
\big[ 
( \bs{\psi_{n+1}}  - \bs{p_0} ) - \beta ( \bs{\psi_{n}}  - \bs{p_0}  ) 
\big].
\end{multline}
Therefore, applying \cite[Corollary~5.3 and Corollary 5.12]{Hairer18},
we obtain
\begin{equation*}
\bs{\overline{\xi}_N}-\bs{p_0}=
\frac{1}{N} \sum_{n=1}^N (\bs{\xi_n} -\bs{p_0})
\stackrel{a.s.}\longrightarrow 
\frac{1-\beta}{\gamma-\beta} 
E[ \bs{{\psi}^{(\pi)}_{n}}  - \bs{p_0} ],
\end{equation*}
where $\big(\bs{{\psi}^{(\pi)}_n}\big)_{n\geq 0}$ is a
Markov chain with transition kernel $P$ and initial distribution
$\pi$. From \eqref{eq-operator-Pid}, we have 
\begin{equation*}
E[ \bs{{\psi}^{(\pi)}_{n}}  - \bs{p_0} ]=
E[ \bs{{\psi}^{(\pi)}_{n+1}}  - \bs{p_0} ]
=E\left[ 
E[ \bs{{\psi}^{(\pi)}_{n+1}}  - \bs{p_0} \,|\, \bs{{\psi}^{(\pi)}_{n}}]
\right]
=
\gamma E[ \bs{{\psi}^{(\pi)}_{n}}  - \bs{p_0} ]
\end{equation*}
and so, since $\gamma<1$, we get $E[ \bs{{\psi}^{(\pi)}_{n}} -
  \bs{p_0} ]=0$. This means that
$\int_S\bs{x}\,\pi(d\bs{x})=\bs{p_0}$ and
$\bs{\overline{\xi}_N}\stackrel{a.s.}\longrightarrow \bs{p_0}$.  \qed
\\

{\bf Proof of the CLT:} We apply Theorem~\ref{CLT-Markov-linear},
taking into account that we have already proven that $\bs{\psi}$ is an
irreducible and aperiodic compact Markov chain.  Since, by
\eqref{eq:f2:K} and what we have already proven before, we have
\[
\bs{\xi_{n+1}} - \bs{p_0}= \bs{f} (\bs{\psi_{n}} , \bs{\psi_{n+1}} )
\quad\text{with}\quad
\bs{f} (\bs{x} , \bs{y} ) = 
\frac{1}{\gamma-\beta} \big[ 
( \bs{y}  - \bs{p_0} ) - 
\beta ( \bs{x}  - \bs{p_0}  ) 
\big] 
\]
and $\bs{p_0}=\int_S\bs{x}\,\pi(d\bs{x})$. Hence $\bs{f}
$ and $P$ form a linear model as defined in
Definition~\ref{def-lin-mod} (see also Remark \ref{rem-lin-mod}).
Indeed, we have $A_{1} = -\frac{\beta}{\gamma-\beta} Id$ and $A_{2} =
\frac{1}{\gamma-\beta} Id$. Moreover, by \eqref{eq-operator-Pid}, we
have $P(id)(\bs{x})-\bs{p_0}=\gamma( \bs{x} -\bs{p_0} )$, which means
$A_{P} = \gamma Id$. Therefore Theorem~\ref{CLT-Markov-linear} holds true 
with
$$
D_0= (1-\gamma)^{-1}Id,\quad 
D_1= -\frac{\gamma(1-\beta)}{(\gamma-\beta)(1-\gamma)}Id,\quad
D_2= \frac{(1-\beta)}{(\gamma-\beta)(1-\gamma)}Id
$$
and so, after some computations, with 
\begin{equation}\label{eq:Sigma2Sigma2pi:K}
\begin{aligned}
\Sigma^2  = 
\frac{(1-\beta)^2(1+\gamma)}{(\gamma-\beta)^2(1-\gamma)}
\Sigma^2_{\pi} = 
\frac{(1-\beta)^2}{(\gamma-\beta)^2}
\Big( 1 + 2
\frac{\gamma}{1-\gamma}
\Big)
\Sigma^2_{\pi} .
\end{aligned}
\end{equation}
In order to conclude, we take a Markov chain
$\big(\bs{{\psi}^{(\pi)}_n}\big)_{n\geq 0}$ with transition kernel
$P$ and initial distribution $\pi$ and we set 
$$
\bs{{\xi}^{(\pi)}_{n+1}}-\bs{p_0}=
f(\bs{{\psi}^{(\pi)}_n},\bs{{\psi}^{(\pi)}_{n+1}})=
A_{2}(\bs{{\psi}^{(\pi)}_{n+1}}-\bs{p_0} ) +
A_{1}(\bs{{\psi}^{(\pi)}_n}-\bs{p_0} ).
$$ Then we observe that, by \eqref{eq:var-cov_psi1} and
\eqref{eq:var-cov_psi0psi1}, we have
\begin{align} \nonumber
\mathrm{diag} (\bs{p_0}) - \bs{p_0}\bs{p_0}^T & =
E\big[(\bs{{\xi}^{(\pi)}_1}-\bs{p_0} )(\bs{{\xi}^{(\pi)}_1}-\bs{p_0} )^\top 
\big]
\\ \nonumber & = 
A_1 \Sigma^2_{\pi} A_1^\top + A_2 \Sigma^2_{\pi} A_2^\top 
+ A_1 \Sigma^2_{\pi} A_P^\top A_2^\top + A_2 A_P \Sigma^2_{\pi} A_1 
\\ &
=
\frac{(\gamma-\beta)^2 + (1-\gamma^2)}{(\gamma-\beta )^2}
\Sigma^2_{\pi}.\label{eq:diagPoPoTSigma2pi:K} 
\end{align}
Finally, it is enough to combine \eqref{eq:Sigma2Sigma2pi:K} and
\eqref{eq:diagPoPoTSigma2pi:K}.
\qed


\subsection{The case $\beta>1$}
\label{sec-beta-maggiore1}

In this case, $\lim_n \beta^{n}=+\infty$ and $\sum_{n\geq
  1}\beta^{-n}= 1/(\beta-1)$. Moreover, by \eqref{eq-rstar_n}, $r_n^*$
increases exponentially to $+\infty$. Hence, the following results
hold true:

\begin{theorem}\label{th-psi}
We have 
$$
\bs{\psi_N}\stackrel{a.s.}\longrightarrow 
\bs{\psi_\infty}= 
\frac{ \bs{B_0} + \alpha \sum_{n=1}^{+\infty} \beta^{-n} \bs{\xi_{n}}}{
|\bs{B_0}| + \frac{\alpha}{\beta-1} }
$$
and 
$$
|\bs{\psi_{N}}-{\bs{\psi_{\infty}}}|=O(\beta^{-N}).
$$ 
Moreover $\bs{\psi_\infty}$ takes values in $\{\bs{x}\in [0,1]^k:\,
|\bs{x}|=1\}$ and, if $B_{0\,i}>0$ for all $i=1,\dots, k$, then
$P\{ \psi_{\infty\, i}\in (0,1)\}=1$ for each $i$.
\end{theorem}
Note that $\bs{\psi_\infty}$ is a function of ${\bs
  \xi}=(\bs{\xi_n})_{n\geq 1}$, which takes values in
$(\{\bs{x}\in\{0,1\}^k:\, |\bs{x}|=1\})^\infty$.
\begin{proof}
By \eqref{eq-psi_n}, we have 
\[
\bs{\psi_N} = \frac{\bs{b_0}+\bs{B_N}}{r_N^*}=
\frac{ 
\bs{b_0}\beta^{-N} + \bs{B_0} + \alpha \sum_{n=1}^N \beta^{-n} \bs{\xi_{n}}
}{
|\bs{b_0}|\beta^{-N} +|\bs{B_0}| + \frac{\alpha }{\beta-1}( 1 - \beta^{-N}) }.
\]
Hence, the almost sure convergence immediately follows because 
$|\sum_{n\geq 1}\beta^{-n} \bs{\xi_{n}}|\leq \sum_{n\geq 1}\beta^{-n}<+\infty$.
Moreover, after some computations, we have 
$$
\bs{\psi_N}-\bs{\psi_\infty}=
\frac{
-\alpha\big(|\bs{B_0}|+\frac{\alpha}{\beta-1}\big)
\sum_{n\geq N+1}\beta^{-n}\bs{\xi_n}+\beta^{-N}\bs{R}
}
{
\big(|\bs{B_0}|+\frac{\alpha}{\beta-1}\big)
\Big(|\bs{B_0}|+\frac{\alpha}{\beta-1}+ 
\beta^{-N}\big(|\bs{b_0}|-\frac{\alpha}{\beta-1}\big)
\Big)
},
$$ where $\bs{R}=
\big(|\bs{B_0}|+\frac{\alpha}{\beta-1}\big)\bs{b_0}
-\big(|\bs{b_0}|-\frac{\alpha}{\beta-1}\big)
\big(\bs{B_0}+\alpha\sum_{n=1}^{+\infty}\beta^{-n}\bs{\xi_n}\big)$.
Therefore, since $|\sum_{n\geq N+1}\beta^{-n}\bs{\xi_n}|\leq
\sum_{n\geq N+1}\beta^{-n}=\beta^{-N}/(\beta-1)$ and $|\bs{R}|$ is
bounded by a constant, we obtain that
$$
|\bs{\psi_N}-\bs{\psi_\infty}|=O(\beta^{-N}).
$$ 
In order to conclude, it is enough to recall that, by definition,
we have $\bs{\psi_N}\in [0,1]^k$ with $|\bs{\psi_N}|=1$ and observe
that, if $B_{0,\, i}>0$ for all $i$, then we have 
$$ 
0<
\frac{B_{0\, i}}
{|\bs{B_0}| + \frac{\alpha}{\beta-1}}
\leq 
\psi_{\infty\, i}= 
\frac{B_{0\, i} + \alpha \sum_{n=1}^\infty \beta^{-n} \xi_{n\,i}}
{|\bs{B_0}| + \frac{\alpha}{\beta-1}}
\leq 
\frac{B_{0\, i} + \frac{\alpha}{\beta-1} }
{|\bs{B_0}| + \frac{\alpha}{\beta-1}}<1. \qedhere
$$
\end{proof}

\begin{theorem}\label{th-xi-media}
We have $\bs{\overline{\xi}_{N}}\stackrel{a.s.}\longrightarrow
{\bs{\psi_{\infty}}}$ and
\begin{equation*}
\sqrt{N}\left(\bs{\overline{\xi}_{N}}-\bs{\psi_{N}}\right)
\mathop{\longrightarrow}^{s}
\mathcal{N}\left(\bs{0}, \Sigma^2 \right)
\quad\mbox{and}\quad
\sqrt{N}\left(\bs{\overline{\xi}_{N}}-{\bs{\psi_{\infty}}}\right)
\mathop{\longrightarrow}^{s}
\mathcal{N}\left(\bs{0}, \Sigma^2\right)
\end{equation*}
where $\Sigma^2 = \mathrm{diag} ({\bs{\psi_{\infty}}}) -
{\bs{\psi_{\infty}}} {\bs{\psi}_\bs{\infty}^\top}$ and
$\mathop{\longrightarrow}\limits^{s}$ means stable convergence.
\end{theorem}
Note that, if $B_{0\,i}>0$ for all $i=1,\dots,k$, then
$\Sigma^2_{i,j}\in (0,1)$ for each pair $(i,j)$.  \\ 

\indent The stable convergence has been introduced in \cite{ren} and,
for its definition and properties, we refer to \cite{crimaldi-libro,
  CriLetPra, HallHeyde}.

\begin{proof} 
The almost sure convergence of $\bs{\overline{\xi}_N}$ to
$\bs{\psi_\infty}$ follows by usual martingale arguments (see, for
instance, \cite[Lemma 2]{BeCrPrRi11}) because
$E[\bs{\xi_{n+1}}|\mathcal{F}_n]=\bs{\psi_n}\to \bs{\psi_\infty}$
a.s. and $\sum_{n\geq 1} E[\|\bs{\xi_n}\|^2 ]n^{-2}\leq \sum_{n\geq 1}
n^{-2}<+\infty$.  \\

Regarding the CLTs, we observe that, by means
of \eqref{dinamica-psin}, we can write
\begin{equation}\label{increment-psi}
\bs{\psi_{n+1}}-\bs{\psi_{n}} =
\frac{\bs{H}(\bs{\psi_{n}})}{{r^*_{n+1}}}+
\frac{\Delta \bs{M_{n+1}} }{{r^*_{n+1}}},
\end{equation}
where $\bs{H}(\bs{x})=(\beta-1)|\bs{b_0}|(\bs{x}-\bs{p_0})$ and
$\Delta\bs{M_{n+1}}=\alpha({\bs{\xi_{n+1}}}-\bs{\psi_{n}})$.  
Therefore, we get 
$$
\begin{aligned}
\sqrt{N}\left(\bs{\overline{\xi}_{N}}-\bs{\psi_{N}}\right) &=
\frac{1}{\sqrt{N}}\left(N\bs{\overline{\xi}_{N}}-N\bs{\psi_{N}}\right)
=
\frac{1}{\sqrt{N}}\sum_{n=1}^N\left[\bs{\xi_{n}}-{\bs{\psi_{n-1}}}
+n({\bs{\psi_{n-1}}}-{\bs{\psi_{n}}})\right]
\\ &
=\sum_{n=1}^N \bs{Y_{N,n}}+\bs{Q_N},
\end{aligned}
$$
where
$$
\bs{Y_{N,n}}=
\frac{\bs{\xi_{n}}-{\bs{\psi_{n-1}}}}{\sqrt{N}}=
\frac{\Delta\bs{M_{n+1}}}{\alpha\sqrt{N}}
$$
and
$$
\bs{Q_N}=
\frac{1}{\sqrt{N}}\sum_{n=1}^N n\left({\bs{\psi_{n-1}}}-{\bs{\psi_{n}}}\right)
=
-\frac{1}{\sqrt{N}}\sum_{n=1}^N 
\frac{n}{r^*_{n}}( \bs{H}({\bs{\psi_{n-1}}}) +
\Delta\bs{M_{n}}).
$$ Since $\sum_{n\geq 1} n/{r^*_{n}}<+\infty$ and
$|\bs{H}({\bs{\psi_{n-1}}})| + |\Delta \bs{M_{n}}|$ is uniformly
bounded by a constant, we have that $\bs{Q_N}$ converges to zero
almost surely. Therefore it is enough to prove that
$\sum_{n=1}^N\bs{Y_{N,n}}$ converges stably to the desired Gaussian
kernel. To this purpose we observe that
$E[\bs{Y_{N,n}}|\mathcal{F}_{n-1}]=\bs{0}$ and so, in order to prove
the stable convergence, we have to check the following conditions (see
\cite[Cor.~7]{CriLetPra} or \cite[Cor.~5.5.2]{crimaldi-libro}):
\begin{itemize}
\item[(c1)] $E\left[\,\max_{1\leq n\leq N} |\bs{Y_{N,n}} |\,\right]\to 0$ and 
\item[(c2)] $\sum_{n=1}^N \bs{Y_{N,n}} {\bs{Y}_{\bs{N,n}}^\top}
\mathop{\longrightarrow}\limits^{P}
\Sigma^2$.
\end{itemize}
Regarding (c1), we observe that $\max_{1\leq n\leq N} |\bs{Y_{N,n}} |
\leq \frac{1}{\sqrt{N}}\max_{1\leq n\leq N} |\bs{\xi_{n}}
-{\bs{\psi_{n-1}}}| \leq \frac{1}{\sqrt{N}} \to 0$.  In order to
conclude, we have to prove condition (c2), that is
$$
\sum_{n=1}^N \bs{Y_{N,n}} {\bs{Y}_{\bs{N,n}}^\top}
=
\frac{1}{N}\sum_{n=1}^N
(\bs{\xi_{n}}-{\bs{\psi_{n-1}}})(\bs{\xi_{n}}-{\bs{\psi_{n-1}}})^\top
\mathop{\longrightarrow}\limits^{P} \Sigma^2.
$$ The above convergence holds true even almost surely by usual
martingale arguments (see, for instance, \cite[Lemma 2]{BeCrPrRi11}).
Indeed, we have $\sum_{n\geq 1}
E[\|\bs{\xi_{n}}-{\bs{\psi_{n-1}}}\|^2]/n^2\leq \sum_{n\geq 1}
n^{-2}<+\infty$ and
$$
E[
(\bs{\xi_{n}}-{\bs{\psi_{n-1}}})(\bs{\xi_{n}}-{\bs{\psi_{n-1}}})^\top
|\mathcal{F}_{n-1}]=
\mathrm{diag} ({\bs{\psi_{n-1}}}) - {\bs{\psi_{n-1}}}
{\bs{\psi_{n-1}}^\top}
\mathop{\longrightarrow}\limits^{a.s} 
\Sigma^2.
$$
The last stable convergence follows from the equality
$$
\sqrt{N}(\bs{\overline{\xi}_{N}}-{\bs{\psi_{\infty}}})=
\sqrt{N}(\bs{\overline{\xi}_{N}}-\bs{\psi_{N}})+
\sqrt{N}(\bs{\psi_{N}}-{\bs{\psi_{\infty}}}),
$$
where the last term converges almost surely to zero by Theorem \ref{th-psi}.
\end{proof}

\begin{remark}\rm 
 Equation \eqref{increment-psi} implies that the bounded stochastic
 processes $\bs{\psi}= (\bs{\psi_{n}})_n$ is a positive
 (i.e.\ non-negative) almost supermartingale \cite{rob} and also a
 quasi-martingale \cite{met}, because $\bs{H}({\bs{\psi_{n}}})$ is
 uniformly bounded by a constant and $\sum_{n\geq 1}
 1/{r^*_{n+1}}<+\infty$.
\end{remark}


\subsection{The case $\alpha=0$}
\label{sec-alpha0}

The model introduced above for $\alpha>0$ makes sense also when
$\alpha=0$. For completeness, in this section we discuss
this case. Recall that we are assuming $|\bs{b_0}|>0$ and
$b_{0\,i}+B_{0\,i}>0$ (see Sec.~\ref{urn-model}). For the case
$\beta>1$, we here assume also $|\bs{B_0}|>0$.  \\

\indent When $\alpha=0$, the random vectors $\bs{\xi_n}$ are
independent with
$$
P(\xi_{n\, i}=1)=\psi_{n\, i}=
\frac{b_{0\,i}+\beta^{n}B_{0\, i}}{|\bs{b_0}|+\beta^{n}|\bs{B_0}|}.
$$ 
Therefore, we have $\psi_{n\,i}=\frac{b_{0\,i}+B_{0\,
    i}}{|\bs{b_0}|+|\bs{B_0}|}$ for all $n$ if $\beta=1$ (which
corresponds to the classical multinomial model) and 
$$
\psi_{n\,i}\longrightarrow
\begin{cases}
\frac{b_{0\,i}}{|\bs{b_0}|}\qquad &\mbox{if } \beta\in [0,1)\\
\frac{B_{0\,i}}{|\bs{B_0}|}\qquad &\mbox{if } \beta>1.
\end{cases}
$$
Moreover, the following result holds true:
\begin{theorem}
We have 
$$
\bs{\overline{\xi}_{N}}\stackrel{a.s.}
\longrightarrow \bs{\overline{\xi}_\infty}=
\begin{cases}
\frac{\bs{b_0}+\bs{B_0}}{|\bs{b_0}|+|\bs{B_0}|} \quad&\mbox{if } \beta=1,\\
\frac{\bs{b_0}}{|\bs{b_0}|}=\bs{p_0} \quad&\mbox{if } \beta\in [0,1),\\
\frac{\bs{B_0}}{|\bs{B_0}|} \quad&\mbox{if } \beta>1.\\
\end{cases}
$$
Moreover, we have 
$$
\sqrt{N}\left(\bs{\overline{\xi}_{N}}-\bs{\overline{\xi}_\infty}\right)
\stackrel{s}\longrightarrow \mathcal{N}(\bs{0},\Sigma^2),
$$
where $\Sigma^2 = \mathrm{diag} ({\bs{\psi_{\infty}}}) -
{\bs{\psi_{\infty}}} {\bs{\psi}_\bs{\infty}^\top}$ and
$\mathop{\longrightarrow}\limits^{s}$ means stable convergence.
\end{theorem}

\begin{proof} The almost sure convergence follows from the Borel–Cantelli 
lemmas (see, for instance, 
\cite[Section 12.15]{williams}). Indeed, we have:
\begin{itemize}
\item if $\sum_{n\geq 0}\psi_{n,i}<+\infty$, then 
$\sum_{n=1}^N\xi_{n\,i}\stackrel{a.s.}\longrightarrow \xi_{\infty\,i}$, with 
$P(\xi_{\infty\,i}<+\infty)=1$;
\item if $\sum_{n\geq 0}\psi_{n,i}=+\infty$, then 
$\sum_{n=1}^N\xi_{n\,i}/\sum_{n=1}^N\psi_{n-1\,i}\stackrel{a.s.}\longrightarrow 1$.
\end{itemize} 
Hence, the statement of the theorem follows because:
\begin{itemize}
\item[(i)] if $\beta=1$, then $\sum_{n\geq
  0}\psi_{n\,i}=+\infty$ and $\sum_{n=1}^{N}\psi_{n-1\,i}\sim
  \frac{b_{0\,i}+B_{0\,i}}{|\bs{b_0}|+|\bs{B_0}|} N$;
\item[(ii)] if  $\beta\in [0,1)$ and $b_{0\,i}>0$, then $\sum_{n\geq
  0}\psi_{n\,i}=+\infty$ and $\sum_{n=1}^{N}\psi_{n-1\,i}\sim
  \frac{b_{0\,i}}{|\bs{b_0}|} N=p_{0\,i} N$;
\item[(iii)] if $\beta\in [0,1)$ and $b_{0\,i}=0$, then $\sum_{n\geq
  0}\psi_{n\,i}\leq \frac{B_{0\,i}}{|\bs{b_0}|}\sum_{n\geq 0}\beta^n
  <+\infty$ and so $\sum_{n\geq 1}\xi_{n\,i}<+\infty$ a.s., that is
  $\xi_{n\,i}=0$ eventually with probability one;
\item[(iv)] if $\beta>1$ and $B_{0\,i}> 0$, then $\sum_{n\geq
  0}\psi_{n\,i}=+\infty$ and $\sum_{n=1}^{N}\psi_{n-1\,i}\sim
  \frac{B_{0\,i}}{|\bs{B_0}|} N$;
\item[(v)] if $\beta>1$ and $B_{0\,i}=0$, then $\sum_{n\geq
  0}\psi_{n\,i}\leq \frac{b_{0\,i}}{|\bs{B_0}|}\sum_{n\geq
  0}\beta^{-n}<+\infty$ and so $\sum_{n\geq 1}\xi_{n\,i}<+\infty$
  a.s., that is $\xi_{n\,i}=0$ eventually with probability one.
\end{itemize}
For the CLT we argue as in the proof of Theorem
\ref{th-xi-media}. Indeed, we set $\bs{Y_{N,n}}=
\frac{\bs{\xi_{n}}-{\bs{\psi_{n-1}}}}{\sqrt{N}}$ so that we have
$$
\sqrt{N}\left(\bs{\overline{\xi}_{N}}-\bs{\overline{\xi}_\infty}\right)=
\sum_{n=1}^N\bs{Y_{N,n}}+
\frac{1}{\sqrt{N}}
\sum_{n=1}^N\big(\bs{\psi_{n-1}}-\bs{\overline{\xi}_\infty}\big),
$$
where the second term converges to zero because 
$$
\sum_{n=1}^N \big| \bs{\psi_{n-1}}-\bs{\overline{\xi}_\infty}\big|
=
\begin{cases}
0\quad&\mbox{if } \beta=1,\\
O\left(\sum_{n=1}^N \beta^n\right)\quad&\mbox{if } \beta\in [0,1),\\
O\left(\sum_{n=1}^N \beta^{-n}\right)\quad&\mbox{if } \beta>1.\\
\end{cases}
$$
Therefore it is enough to prove that
$\sum_{n=1}^N\bs{Y_{N,n}}$ converges stably to the desired Gaussian
kernel. To this purpose we observe that
$E[\bs{Y_{N,n}}|\mathcal{F}_{n-1}]=\bs{0}$ and so, in order to prove
the stable convergence, we have to check the following conditions (see
\cite[Cor.~7]{CriLetPra} or \cite[Cor.~5.5.2]{crimaldi-libro}):
\begin{itemize}
\item[(c1)] $E\left[\,\max_{1\leq n\leq N} |\bs{Y_{N,n}} |\,\right]\to 0$ and 
\item[(c2)] $\sum_{n=1}^N \bs{Y_{N,n}} {\bs{Y}_{\bs{N,n}}^\top}
\mathop{\longrightarrow}\limits^{P}
\Sigma^2$.
\end{itemize}
Regarding (c1), we note that $\max_{1\leq n\leq N} |\bs{Y_{N,n}} |
\leq \frac{1}{\sqrt{N}}\max_{1\leq n\leq N} |\bs{\xi_{n}}
-{\bs{\psi_{n-1}}}| \leq \frac{1}{\sqrt{N}} \to 0$.  Regarding
condition (c2), we observe that 
$$
\sum_{n=1}^N \bs{Y_{N,n}} {\bs{Y}_{\bs{N,n}}^\top}
=
\frac{1}{N}\sum_{n=1}^N
(\bs{\xi_{n}}-{\bs{\psi_{n-1}}})(\bs{\xi_{n}}-{\bs{\psi_{n-1}}})^\top
\mathop{\longrightarrow}\limits^{a.s.} \Sigma^2,
$$ 
because (see, for instance, \cite[Lemma 2]{BeCrPrRi11})
$\sum_{n\geq 1} E[\|\bs{\xi_{n}}-{\bs{\psi_{n-1}}}\|^2]/n^2\leq
\sum_{n\geq 1} n^{-2}<+\infty$ and
$$
E[
(\bs{\xi_{n}}-{\bs{\psi_{n-1}}})(\bs{\xi_{n}}-{\bs{\psi_{n-1}}})^\top
|\mathcal{F}_{n-1}]=
\mathrm{diag} ({\bs{\psi_{n-1}}}) - {\bs{\psi_{n-1}}}
{\bs{\psi_{n-1}}^\top}
\mathop{\longrightarrow}\limits^{a.s} 
\Sigma^2.
$$
\end{proof}


\section{Proof of the goodness of fit result 
(Theorem \ref{th-chi-squared-test})}
\label{proof-chi-squared-test}
The proof is based on Theorem \ref{thm:CLT_Beta0:K} (for $\beta = 0$)
and Theorem~\ref{thm:CLT_beta<1} (for $0<\beta<1$), whose proofs are
in Sections \ref{sec-beta0} and \ref{sec-beta-minore1},
respectively. The almost sure convergence of $O_i/N$ immediately
follows since $O_i/N=\overline{\xi}_{N\,i}$. In order to prove the
stated convergence in distribution, we mimic the classical proof for
the Pearson chi-squared test based on Sherman Morison formula (see
\cite{SM50}), but see also \cite[Corollary~2]{RaoScott81}.

\begin{proof}
We start recalling the Sherman Morison formula: if $A$ is an
invertible square matrix and $ 1 - \bs{v}^\top A^{-1} \bs{u} \neq 0 $,
then
\[
(A - \bs{u}\bs{v}^\top)^{-1} =  
A^{-1} + \frac{A^{-1} \bs{u}\bs{v}^\top A^{-1}}{1 - \bs{v}^\top A^{-1} \bs{u} }.
\]
Given the observation
$\bs{\xi_{n}}=(\xi_{n\,1},\dots,\xi_{n\,k})^{\top}$, we define the
``truncated'' vector $\bs{\xi^*_{n}}= (\xi^*_{n\, 1}, \ldots,
\xi^*_{n\, k-1})^\top$, given by the first $k-1$ components of
$\bs{\xi_{n}}$. Theorem~\ref{thm:CLT_Beta0:K} (for
$\beta = 0$) and Theorem~\ref{thm:CLT_beta<1} (for $\beta \in (0,1)$) give 
the Central Limit Theorem for $(\bs{\xi_{n}})_n$, that immediately
implies 
\begin{equation}\label{eq:truncCLT}
\sqrt{N}\left(\bs{\overline{\xi}^*_{N}}-\bs{p^*}\right)
=
\frac{\sum_{n=1}^N (\bs{\xi^*_n}  - \bs{p^*}) }{\sqrt{N}} 
\mathop{\longrightarrow}^{d} 
\mathcal{N} (\bs{0},\Sigma_*^2),
\end{equation}
where $\bs{p^*}$ is given by the first $k-1$ components of $\bs{p_0}$
and 
\[ 
\Sigma_*^2 = \lambda ( \mathrm{diag} (\bs{p^*}) - \bs{p^*}\bs{p^*}^T).
\]  
By assumption $p_{0\,i}>0$ for all $i=1,\dots,k$ and so $\mathrm{diag}
(\bs{p^*})$ is invertible with inverse $\mathrm{diag} (\bs{p^*})^{-1}
= \mathrm{diag} (\frac{1}{p_{0\,1}}, \ldots, \frac{1}{p_{0\,k-1}} )$
and, since $ (\mathrm{diag} (\bs{p^*})^{-1} ) \bs{p^*} =
\bs{1}\in\mathbb{\R}^{k-1}$, we have
\[
1 -  \bs{p^*}^T \mathrm{diag} (\bs{p^*})^{-1} \bs{p^*} = 
1-\sum_{i=1}^{k-1}p_{0\,i}=
\sum_{i=1}^k p_{0\,i} - \sum_{i=1}^{k-1} p_{0\,i} = p_{0\,k} >0.
\]
Therefore we can use the Sherman Morison formula with $A =
\mathrm{diag} (\bs{p^*})$ and $\bs{u}= \bs{v} = \bs{p^*}$, and we
obtain
\begin{equation}\label{eq:SMformula}
(\Sigma_*^2)^{-1} = \frac{1}{\lambda} 
( \mathrm{diag} (\bs{p^*}) - \bs{p^*}\bs{p^*}^T)^{-1}
=
\frac{1}{\lambda} 
\Big( \mathrm{diag} (\tfrac{1}{p_{0\,1}}, \ldots, \tfrac{1}{p_{0\,k-1}} ) 
+ \frac{1}{p_{0\,k}} \bs{1}\bs{1}^\top \Big).
\end{equation}
Now, since $\sum_{i=1}^{k} (\overline{\xi}_{N\,i} - {p_{0\,i}}) = 0$, 
then $\overline{\xi}_{N\,k} - {p_{0\,k}} = \sum_{i=1}^{k-1}
(\overline{\xi}_{N\,i} - {p_{0\,i}})$ and so we get
\[
\begin{aligned}
\sum_{i=1}^{k} \frac{(O_{i} - N{p_{0\,i}})^2}{N{p_{0\,i}}} & = 
N \sum_{i=1}^{k} \frac{(\overline{\xi}_{N\,i} - {p_{0\,i}})^2}{{p_{0\,i}}} 
= 
N \Big[ 
\sum_{i=1}^{k-1} \frac{(\overline{\xi}_{N\,i} - {p_{0\,i}})^2}{{p_{0\,i}}} + 
\frac{(\overline{\xi}_{N\,k} - {p_{0\,k}})^2}{{p_{0\,k}}} \Big]
\\
& = 
N \Big[ \sum_{i=1}^{k-1} \frac{(\overline{\xi}_{N\,i} - {p_{0\,i}})^2}{{p_{0\,i}}} + 
\frac{(\sum_{i=1}^{k-1} (\overline{\xi}_{N\,i} - {p_{0\,i}}) )^2}{{p_{0\,k}}} 
\Big]
\\
& = 
N \sum_{i_1,i_2=1}^{k-1} (\overline{\xi}_{N\,{i_1}} - {p_{0\,i_1}}) 
(\overline{\xi}_{N\,i_2} - {p_{0\,i_2}}) 
\Big( \delta_{i_1}^{i_2}\frac{1}{{p_{0\,i_1}}} + 
\frac{1}{{p_{0\,k}}} \Big),
\end{aligned}
\]
where $\delta_{i_1}^{i_2}$ is equal to $1$ if $i_1=i_2$ and equal to
zero otherwise.  Finally, from the above equalities, recalling
\eqref{eq:truncCLT} and \eqref{eq:SMformula}, we obtain
\begin{equation*}
\sum_{i=1}^{k} \frac{(O_{i} - N{p_{0\,i}})^2}{N{p_{0\,i}}} =
\lambda 
N (\bs{\overline{\xi}^*_N}  - \bs{p^*})^\top 
(\Sigma_*^2)^{-1} (\bs{\overline{\xi}^*_N}  - \bs{p^*})
\mathop{\longrightarrow}\limits^{d}
\lambda W_0=W_*,
\end{equation*}
where $W_0$ is a random variable with distribution
$\chi^2(k-1)=\Gamma((k-1)/2,1/2)$, where $\Gamma(a,b)$ denotes the
Gamma distribution with density function
$$
f(w)=\frac{b^a}{\Gamma(a)}w^{a-1}e^{-bw}.
$$
As a consequence, $W_*$ has distribution $\Gamma((k-1)/2,1/(2\lambda))$.
\end{proof}

\noindent {\bf Declaration}\\
\noindent All the authors developed the theoretical results, performed
the numerical simulations, contributed to the final version of the
manuscript.

\appendix

\numberwithin{equation}{section}

\section{}

\subsection{A central limit theorem for a multidimensional compact Markov 
chain}
\label{sec-clt-markov}

In this section we prove the general Central Limit Theorem for Markov
chains, used for the proof of Theorem \ref{thm:CLT_beta<1}.  \\

\indent Let $(S,d)$ be a {\em compact metric space} and denote by
$C(S)$ the space of continuous real functions on $S$, by $Lip(S)$ the
space of Lipschitz continuous real functions on $S$ and by
$Lip(S\times S)$ the space of Lipschitz continuous real functions on
$S\times S$.  Moreover, we define $\|f\|_\infty=\sup_{x\in S}|f(x)|$
for each $f$ in $C(S)$ and, for each $f$ in $Lip(S)$,
\[
|f|_{Lip} = \sup_{x,y\in S,\, x\neq y} 
\frac{|f(y)-f(x)|}{d(x,y)} 
\qquad\mbox{and }\qquad
\|f\|_{Lip}=|f|_{Lip}+\|f\|_\infty.
\]

Let $P(x,dy)$ be a {\em Markovian kernel} on $S$ and set $(Pf)(x) = \int_S
f(y)P(x,dy)$. We now recall some definitions and results regarding Markov
chains with values in $S$. 

\begin{definition}\label{weak-Feller}
We say that $P$ is \emph{weak Feller} if $(Pf)(x)=\int_S f(y)P(x,dy)$
defines a linear operator $P:C(S)\to C(S)$. A Markov chain with a weak
Feller transition kernel is said a weak Feller Markov chain.
\end{definition}

\begin{remark}\label{remark-existence-invariant-measure}
If $P$ is weak Feller, then the sequence $(P^n)_{n\geq 1}$ of
operators from $C(S)$ to $C(S)$ is uniformly bounded with respect to
$\|\cdot\|_\infty$: indeed, we simply have
\begin{multline*}
\|P^nf\|_\infty=\sup_{x\in S} |P^nf(x)| = 
\sup_{x\in S} \Big|\int_S f(y)P^n(x,dy)\Big|
\\
\leq \sup_{x\in S} \Big( \int_S \sup_{y\in S} |f(y) | P^n(x,dy) \Big) =
\sup_{y\in S} |f(y)|=\|f\|_\infty .
\end{multline*}
Moreover, the existence of at least one invariant probability measure
for $P$ is easily shown.  In fact, the set of probability measures
$\mathcal{P}(S)$ on $S$, endowed with the topology of the weak
convergence, is a compact convex set.  In addition, the adjoint
operator of $P$, namely
\[
P^* : \mathcal{P}(S) \to \mathcal{P}(S) , 
\qquad (P^*\nu)(B) = \int_S \nu(dx) P(x,B) ,
\] 
is continuous on $\mathcal{P}(S)$ (since $P$ is weak Feller).  Then,
the existence of an invariant probability measure $\pi$ is a
consequence of the Brouwer's fixed-point theorem.
\end{remark}

\begin{definition}\label{def-star}
We say that $P$ is \emph{semi-contractive} or a \emph{semi-contraction} on 
$Lip(S)$ if it maps $Lip(S)$ into itself and there exists a constant
$\gamma<1$ such that
$$
|Pf|_{Lip}\leq \gamma |f|_{Lip}
$$
for each $f\in Lip(S)$.
\end{definition}

We now give the definition of \emph{compact Markov chain} (see
\cite[Chapter~3]{Nor72} for a general exposition of the theory of
these processes, and \cite{Doeb37} for the beginning of this theory):

\begin{definition}\label{def:compact-Markov-process}
We say that $P$ is a \emph{Doeblin-Fortet operator} if it is weak
Feller, a bounded operator from $(Lip(S),\|\cdot\|_{Lip})$ into itself
and there are finite constants $n_0\geq 1$, $\gamma<1$ and $R\geq 0$
such that
$$
|P^{n_0}f|_{Lip}\leq \gamma |f|_{Lip} + R \|f\|_\infty,
$$ for each $f\in Lip(S)$. A Markov chain with a Doeblin-Fortet
operator on a compact set $S$ is called \emph{compact Markov chain} (or
\emph{process}).
\end{definition}
Note that the Doeblin-Fortet operator, the weak Feller property and
the semi-contraction may also be defined for not-compact state space.
In general, a compact Markov process is a Doeblin-Fortet process in a
compact state space. In our framework, since $S$ is compact, the two
concepts coincide and the following result follows immediately:

\begin{lemma}\label{lem:cMP}
If $P$ is weak Feller and a semi-contractive operator on $Lip(S)$, then $P$
is a Doeblin-Fortet operator. In other words, a weak Feller Markov
chain such that its transition kernel is semi-contractive on $Lip(S)$ is a
compact Markov chain.
\end{lemma}

\begin{definition}\label{def-irreducible}
We say that $P$ is \emph{irreducible} and \emph{aperiodic} if 
\begin{equation*}
Pf=e^{i\theta}f,\quad \text{with } 
\theta\in\mathbb{R},\quad f\in Lip(S)
\Rightarrow 
e^{i\theta}=1\quad\mbox{and}\quad f=\mbox{constant}.
\end{equation*}
A Markov chain with an irreducible and aperiodic transition kernel is
said an irreducible and aperiodic Markov chain.
\end{definition}

Under the hypotheses of the Theorem of Ionescu-Tulcea and Marinescu in
\cite{IonMar50}, the spectral radius of $P$ is $1$, the set of
eigenvalues of $P$ of modulus $1$ has only a finite number of elements
and each relative eigenspace is finite dimensional. This theorem can
always be applied to a compact Markov chain (see
\cite[Theorem~3.3.1]{Nor72}).  More specifically, every compact Markov
chain has $d$ disjoint closed sets, called \emph{ergodic sets},
contained in its compact state space $S$.  These sets are both the
support of the base of the ergodic invariant probability measures, and
the support of a base of the eigenspaces related to the eigenvalues of
modulus $1$ (see \cite[Theorem~3.4.1]{Nor72}).  In addition, each of
this ergodic set may be subdivided into $p_j$ closed disjoint
subsets. The number $p_j$ is the period of the $j$-th irreducible
component, and the ergodic subdivision gives the support of the
eigenfunctions related to the $p_j$ roots of $1$ (see
\cite[Theorem~3.5.1]{Nor72}). Then, as also explained in
\cite[\S~3.6]{Nor72}, there are not other eigenvalues of modulus $1$
except $1$ (aperiodicity) and not other eigenfunctions except the
constant for the eigenvalue equal to $1$ (irriducibility) if and only
if the compact Markov chain has but one ergodic kernel, and this
kernel has period $1$.  In other words, the following result holds
true:

\begin{theorem}\label{th-irreducible}
Let $\psi = (\psi_n)_{n\geq 0}$ be a compact Markov chain and let
$\pi$ an invariant probability measure with respect to its transition
kernel. If $\psi = (\psi_n)_{n\geq 0}$ converges in distribution to
$\pi$, whatever is its initial distribution, then $\pi$ is the unique
invariant probability measure and $\psi$ is irreducible and aperiodic.
\end{theorem}

We now note that, if $P$ is Doeblin-Fortet, irreducible and aperiodic,
then it satisfies all the conditions given in \cite[D\'efinition
  0]{Guiv88} and \cite[D\'efinition 1]{Guiv88}.  Therefore, it has a
{\em unique} invariant probability measure $\pi$ and, for any $f\in
Lip(S\times S)$, there exists a unique (up to a constant) function
$u_f \in Lip(S)$ such that
\[
u_f(x) - Pu_f(x) =  
\int_S f(x,y)P(x,dy) - \int_S \int_S f(x,y)P(x,dy) \pi(dx).
\]
By means of this function $u_f$, it is possible to define the (unique)
function \( f'(x,y) = f(x,y)+u_f(y)-u_f(x) \) so that we have
\begin{equation*}
m(f) = \int_S \int_S f(x,y)P(x,dy) \pi(dx)=
\int_S \int_S f'(x,y)P(x,dy) \pi(dx) = m(f').
\end{equation*}
In addition, we may define the quantity $\sigma^2(f) \geq 0$ as (see
{\cite[Eq.~(6)]{Guiv88}})
\begin{multline}\label{eq-def-sigma}
\sigma^2(f) = 
\int_S \int_S \big[ f'(x,y)-m(f') \big]^2 P(x,dy) \pi(dx)
\\
= 
\int_S \int_S \big[ f(x,y)-m(f)+u_f(y)-u_f(x) \big]^2 P(x,dy) \pi(dx).
\end{multline}

Finally, we have the following convergence result:

\begin{theorem}[{\cite[Th\'eor\'eme~1 and Th\'eor\'eme~2]{Guiv88}}] 
\label{CLT-Markov}
Let $\psi=(\psi_n)_{n\geq 0}$ be an irreducible and aperiodic compact
Markov chain and denote by $\pi$ its unique invariant probability
measure.  Let $f\in Lip(S\times S)$ such that $m(f)=0$ and
$\sigma^2(f)>0$.  Then, setting $S_N(f)=\sum_{n=0}^{N-1}
f(\psi_n,\psi_{n+1})$, we have
\begin{equation*}
\frac{S_N(f)}{\sqrt{N}} 
\mathop{\longrightarrow}^{d}_{N\to\infty} 
\mathcal{N}\big(0,\sigma^2(f)\big),
\end{equation*}
and
\begin{equation*}
\sup_{t}
\Big| P\big(S_N(f)<t\sqrt{N}\big)- 
\mathcal{N}\big(0,\sigma^2(f)\big)(-\infty, t)\Big|  
=O(1/\sqrt{N}).
\end{equation*}
\end{theorem} 

Now, let us specialize our assumptions taking as $S$ a compact subset
of $\mathbb{R}^k$. Therefore, in the sequel we will use the boldface
in order to highlight the fact the we are working with vectors.
 
\begin{definition}[``Linearity'' condition]\label{def-lin-mod} 
We say that $P$ and $\bs{f}: S\times S
  \to \mathbb{R}^d$ form a \emph{linear model} if $\bs{f}$ is linear
  (in $\bs{x}$ and $\bs{y}$) with \(m(\bs{f})= \bs{0}\) and the
  function
$$
(P\bs{y})(\bs{x})= \int_S\bs{y}P(\bs{x},\bs{dy})
$$
is linear (in $\bs{x}$).
\end{definition}

\begin{remark}\label{rem-lin-mod}
\rm 
Denote by $\bs{p_0} = \int_S \bs{x} \pi(\bs{dx})$ the mean value under
the invariant probability measure $\pi$ of $P$. If $P$ and $\bs{f}$
form a linear model, then there exist two matrices $A_{1},A_{2} \in
\mathbb{R}^{d\times k}$ such that
\begin{equation}\label{def-matriciA}
\bs{f} (\bs{x},\bs{y}) = A_{1}(\bs{x}-\bs{p_0}) + A_{2}(\bs{y}-\bs{p_0})
\end{equation}
and a square matrix $A_{P} \in \mathbb{R}^{k\times k}$ such that
\begin{equation}\label{def-matriceA_P}
(P(\bs{y}-\bs{p_0})) (\bs{x}) = \int_S (\bs{y}-\bs{p_0}) P(\bs{x},\bs{dy}) = 
A_{P} (\bs{x}-\bs{p_0}).
\end{equation}
Indeed, if \((P\bs{y})(\bs{x}) = A_{P}\bs{x} + \bs{b}\), using that $\pi$
is invariant with respect to $P$, we obtain
\[
\bs{p_0} = \int_S \bs{y} \pi(\bs{dy}) = 
\int_S \int_S \bs{y} P(\bs{x},\bs{dy}) \pi(\bs{dx}) = 
\int_S [A_{P}\bs{x} + \bs{b}]\pi(\bs{dx}) = A_{P}\bs{p_0} + \bs{b},
\]
and hence $(P(\bs{y}-\bs{p_0})) (\bs{x}) = A_{P} (\bs{x}-\bs{p_0})$.
Moreover, if $\bs{f}(\bs{x},\bs{y}) = A_{1}\bs{x}+A_{2}\bs{y}+\bs{b}$,
then
\begin{align*}
m(A_{1}\bs{x}+ A_{2}\bs{y} + \bs{b}) 
& = 
\int_S P(\bs{x},\bs{dy}) \int_S
A_{1}\bs{x} \pi(\bs{dx})
+
\int_S A_{2}\bs{y}  \pi(\bs{dy}) + \bs{b}
\\
& = (A_{1}+ A_{2})\bs{p_0} + \bs{b}
\end{align*}
and hence, if $m(\bs{f})= \bs{0}$, we obtain $\bs{f}=
A_{1}(\bs{x}-\bs{p_0}) + A_{2}(\bs{y}-\bs{p_0})$.
\end{remark}

\begin{theorem}\label{CLT-Markov-linear}
Let $\psi=(\psi_n)_{n\geq 0}$ be an irreducible and aperiodic compact
Markov chain and denote by $P$ its transition kernel and by $\pi$ its
unique invariant measure. Assume that $P$ and $\bs{f}$ form a linear
model and let $A_1,\,A_2$ and $A_P$ defined as in \eqref{def-matriciA}
and in \eqref{def-matriceA_P}. Then, setting
$\bs{S_N}(\bs{f})=\sum_{n=0}^{N-1}
\bs{f}(\bs{\psi_n},\bs{\psi_{n+1}})$, we have
\begin{equation*}
\frac{\bs{S_N}(\bs{f})}{\sqrt{N}} 
\mathop{\longrightarrow}^{d}_{N\to\infty} 
\mathcal{N}\big(\bs{0},\Sigma^2\big),
\end{equation*}
where
\[
\Sigma^2 = D_{1} \Sigma^2_{\pi} D_{1}^\top+D_{1} \Sigma^2_{\pi} A_{P}^\top D_{2}^\top +
D_{2} A_{P} \Sigma^2_{\pi} D_{1}^\top+D_{2} \Sigma^2_{\pi} D_{2}^\top ,
\]
with 
$$
\Sigma^2_\pi = \int_S (\bs{x}-\bs{p_0})(\bs{x}-\bs{p_0})^\top
\pi(\bs{dx})
$$ 
(the variance-covariance matrix under the invariant
probability measure $\pi$),
\[
D_{1} = A_{1}-D_0\quad\mbox{and}\quad D_{2} = A_{2}+D_0,
\]
where $D_0 = (A_{1}+ A_{2} A_{P})(Id - A_{P})^{-1} $. 
Moreover, for any $\bs{c}\in\mathbb{R}^k$,
\begin{equation*}
\sup_{t}
\Big| P\big(\bs{S_N}(\bs{c}^\top \bs{f})<t\sqrt{N}\big)- 
\mathcal{N}\big(0,\bs{c}^\top\Sigma^2\bs{c}\big)(-\infty, t)\Big|  
=O(1/\sqrt{N}).
\end{equation*}
\end{theorem} 

\begin{proof}
As a consequence of Definition~\ref{def-star}, the spectral radius
of $A_{P}$ must be less than one, and hence $Id - A_{P}$ is
invertible. Therefore, we may define 
\[
\bs{u_f}(\bs{x}) = D_0(\bs{x}-\bs{p_0}) = 
(A_{1}+ A_{2} A_{P})(Id - A_{P})^{-1} (\bs{x}-\bs{p_0}) ,
\]
so that we have 
\begin{align*}
\bs{u_f} (\bs{x}) - (P\bs{u_f})(\bs{x}) = & 
(A_{1}+ A_{2} A_{P})(Id - A_{P})^{-1} (\bs{x}-\bs{p_0})
\\
& \qquad - 
(A_{1}+ A_{2} A_{P})(Id - A_{P})^{-1} A_{P} (\bs{x}-\bs{p_0})
\\
= & 
(A_{1}+ A_{2} A_{P}) (\bs{x}-\bs{p_0})
\\
= &
A_{1} (\bs{x}-\bs{p_0}) \int_S P(\bs{x},\bs{dy}) 
+ A_{2} 
\int_S (\bs{y}-\bs{p_0}) P(\bs{x},\bs{dy}) 
\\
= &
\int_S \bs{f}(\bs{x},\bs{y}) P(\bs{x},\bs{dy}) - \bs{0} \\
= &
\int_S \bs{f}(\bs{x},\bs{y}) P(\bs{x},\bs{dy}) - 
\int_S \bs{f}(\bs{x},\bs{y}) P(\bs{x},\bs{dy}) \pi(\bs{dx}) .
\end{align*}
We immediately get that the function \( \bs{g}(\bs{x},\bs{y}) =
\bs{f}(\bs{x},\bs{y}) +\bs{u_f} (\bs{y}) -\bs{u_f} (\bs{x}) \) is
linear and it may be written as \( \bs{g}(\bs{x},\bs{y}) =
D_{1}(\bs{x}-\bs{p_0}) + D_{2}(\bs{y}-\bs{p_0})\).  Taking into
account that
\[
\int_S \int_S (\bs{y}-\bs{p_0}) P(\bs{x},\bs{dy}) \pi(\bs{dx}) = 
\int_S (\bs{y}-\bs{p_0}) \pi(\bs{dy}) = \bs{0},
\] 
\begin{equation}\label{eq:var-cov_psi1}
\int_S \int_S (\bs{y}-\bs{p_0})(\bs{y}-\bs{p_0})^\top P(\bs{x},\bs{dy}) 
\pi(\bs{dx}) = 
\int_S (\bs{y}-\bs{p_0})(\bs{y}-\bs{p_0})^\top \pi(\bs{dy}) = \Sigma^2_{\pi}
\end{equation}
and
\begin{equation}\label{eq:var-cov_psi0psi1}
\int_S \int_S (\bs{y}-\bs{p_0})(\bs{x}-\bs{p_0})^\top P(\bs{x},\bs{dy}) 
\pi(\bs{dx}) = 
A_{P} 
\int_S (\bs{x}-\bs{p_0})(\bs{x}-\bs{p_0})^\top \pi(\bs{dx}) = 
A_{P} \Sigma^2_{\pi},
\end{equation}
we can compute the quantity
\begin{multline*}
\int_S \int_S \bs{g}(\bs{x},\bs{y}) 
\bs{g}(\bs{x},\bs{y})^\top
P(\bs{x},\bs{dy}) \bs{\pi_\psi}(\bs{dx}) 
\\
\begin{aligned}
& = \int_S \int_S \big[ D_{1}(\bs{x}-\bs{p_0})
+ D_{2}(\bs{y}-\bs{p_0}) \big] \\
& \qquad \qquad \qquad
\big[ D_{1}(\bs{x}-\bs{p_0})
+ D_{2}(\bs{y}-\bs{p_0}) \big]^\top
P(\bs{x},\bs{dy}) \bs{\pi_\psi}(\bs{dx}) 
\\
& = D_{1} \Sigma^2_{\pi} D_{1}^\top+D_{1} \Sigma^2_{\pi} A_{P}^\top D_{2}^\top +
D_{2} A_{P} \Sigma^2_{\pi} D_{1}^\top+D_{2} \Sigma^2_{\pi} D_{2}^\top 
\\
&= \Sigma^2.
\end{aligned}
\end{multline*}
By the Cram\'er-Wold device, the theorem is proven with $\Sigma^2$
given above if we prove that, for any $\bs{c}$,
\begin{equation*}
\bs{c}^\top\frac{\bs{S_N}(\bs{f})}{\sqrt{N}} 
=
\frac{\bs{S_N}(\bs{c}^\top\bs{f})}{\sqrt{N}} 
\mathop{\longrightarrow}^{d}_{N\to\infty} 
\mathcal{N}\big(0,\bs{c}^\top\Sigma\bs{c}\big).
\end{equation*}
Therefore, in order to conclude, it is enough to note that the above
convergence is a consequence of Theorem~\ref{CLT-Markov} with $f =
\bs{c}^\top\bs{f}$. Indeed, by definition $f\in Lip(S\times S)$ and
the function $u_f\in Lip(S)$ in \eqref{eq-def-sigma} may be chosen as
$ u_f = \bs{c}^\top\bs{u_f}$, so that $m(f)=0 $ and $ \sigma^2(f)=
\bs{c}^\top\Sigma\bs{c}$.
\end{proof}


\subsection{Coupling technique}
\label{sec-coupling}

The result proven in this subsection plays a relevant r\^ole in the
proof of Theorem~\ref{thm:CLT_beta<1}. Indeed, it shows that, under
suitable assumptions, two stochastic processes can be ``coupled'' in a
suitable way, preserving their respective joint distributions.  
\\

\indent Set $S^*=\{\bs{x}:\,x_i\geq 0,\; |\bs{x}|= 1\} $, that is the
standard (or probability) simplex in $\mathbb{R}^k$, and recall that
\(\{\bs{e_1}, \ldots, \bs{e_k}\}\) denotes the canonical base of
$\mathbb{R}^k$. We have the following technical lemma:

\begin{lemma}\label{lem:Py_coupling}
There exist two measurable functions $\bs{h^{(1)}},\bs{h^{(2)}}
:S^*\times S^*\times (0,1) \to \{\bs{e_1}, \ldots, \bs{e_k}\}$, such
that for any ${\bs{x} ,\bs{y} \in S^*}$
\begin{equation}\label{eq:lemcoupl_def}
\begin{aligned}
&\int_{(0,1)} \ind{\{\bs{h^{(1)}} (\bs{x} , \bs{y}, u) = \bs{e_i}\}} du = x_i, 
&& \forall i = 1, \ldots, k,
\\
&\int_{(0,1)} \ind{\{\bs{h^{(2)}} (\bs{x} , \bs{y}, u) = \bs{e_i}\}} du = y_i,  
&& \forall i = 1, \ldots, k
\end{aligned}
\end{equation}
and
\begin{equation}\label{eq:lemcoupl_coupl}
\int_{(0,1)} \ind{\{\bs{h^{(1)}} (\bs{x} , \bs{y}, u) \neq 
\bs{h^{(2)}} (\bs{x} , \bs{y}, u) \}} du 
\leq \frac{|\bs{x} -\bs{y}|}{2}. 
\end{equation}
As a consequence, we have 
\begin{equation}\label{eq:lemcoupl_coupl2}
\int_{(0,1)} | \bs{h^{(1)}} (\bs{x} , \bs{y}, u) - 
\bs{h^{(2)}} (\bs{x} , \bs{y}, u) | du 
\leq |\bs{x} -\bs{y}|. 
\end{equation}
\end{lemma}

\begin{proof}
Given $\bs{x},\bs{y} \in S^*$, define $\bs{\underline{xy}}= \bs{x}
\wedge \bs{y}$ so that $\underline{xy}_{i}= \min(x_{i},y_i)$. Set $u_0
= |\bs{\underline{xy}}| = \sum_{i=1}^k \min(x_{i},y_i)$, and note that
$0 \leq u_0 \leq 1$. Moreover, for any $i \in \{1,\ldots,k\}$, set
\[
\begin{aligned}
A_{\bs{\underline{xy}}\,i} & = \Big\{u \colon \sum_{j=1}^{i-1} \underline{xy}_{j} < u 
\leq 
\sum_{j=1}^{i} \underline{xy}_{j}\Big\},
\\
A_{\bs{x}\,i} & = \Big\{u \colon u_0+\sum_{j=1}^{i-1} (x_j-\underline{xy}_{j}) < u 
\leq 
u_0+ \sum_{j=1}^{i} (x_j -\underline{xy}_{j}) \Big\}
\\
A_{\bs{y}\,i} & = \Big\{u \colon u_0+\sum_{j=1}^{i-1} (y_j-\underline{xy}_{j}) < u 
\leq 
u_0+ \sum_{j=1}^{i} (y_j -\underline{xy}_{j}) \Big\}
\end{aligned}
\]
and let 
  \[
\begin{aligned}
  \bs{h^{(1)}} (\bs{x} , \bs{y}, u) = \bs{e_i} , \qquad \text{if } u \in 
  A_{\bs{\underline{xy}}\,i} \cup A_{\bs{x}\,i}\quad\mbox{and }
  \\
  \bs{h^{(2)}} (\bs{x} , \bs{y}, u) = \bs{e_i} , \qquad \text{if } u \in 
  A_{\bs{\underline{xy}}\,i} \cup A_{\bs{y}\,i} .
\end{aligned}
  \]
Observe that, since $1 = u_0 + \sum_{i=1}^{k} (x_i-\underline{xy}_{i})
= u_0 + \sum_{i=1}^{k} (y_i-\underline{xy}_{i})$, the equalities above
uniquely define $\bs{h^{(1)}},\bs{h^{(2)}}$ on the whole domain.
Moreover, since $x_i =\underline{xy}_{i} +(x_i-\underline{xy}_{i})$
and $y_i = \underline{xy}_{i}+(y_i-\underline{xy}_{i})$, then the two
conditions collected in Equation \eqref{eq:lemcoupl_def} are verified.
\\ \indent To check\eqref{eq:lemcoupl_coupl}, just note that
$\bs{h^{(1)}} (\bs{x} , \bs{y}, u) $ is equal to $\bs{h^{(2)}} (\bs{x}
, \bs{y}, u) $ on the set $\cup_i A_{\bs{\underline{xy}}\,i} =
(0,u_0)$ and we have
$$ 
\begin{aligned}
2(1 - u_0) & = \sum_{i=1}^k x_i - \sum_{i=1}^k y_i - 
2 \sum_{i=1}^k\underline{xy}_{i}
= \sum_{i=1}^k ( x_i + y_i - \min(x_i,y_i) ) - \min(x_i,y_i) 
\\
&  =
\sum_{i=1}^k \max(x_i,y_i) - \min(x_i,y_i) = |\bs{x} -\bs{y}| .
\end{aligned}
$$ 
Finally, \eqref{eq:lemcoupl_coupl2} follows immediately
from \eqref{eq:lemcoupl_coupl} since $| \bs{h^{(1)}}-\bs{h^{(2)}} |
\leq 2$.
\end{proof}

Now, we are ready to prove the following ``coupling result'':

\begin{theorem}\label{th:coupl_processes}
Let $\bs{{\psi}^{(1)}}= (\bs{{\psi}_{n}^{(1)}})_n$ and $\bs{{\psi}^{(2)}}=
(\bs{{\psi}_{n}^{(2)}})_n$ be two stochastic processes with values in
$S^*$ that evolve according to the following dynamics:
\begin{equation}\label{eq:evolution_psi:K_modified}
\begin{aligned}
\bs{{\psi}_{n+1}^{(\ell)}} & = a_0 \bs{{\psi}_{n}^{(\ell)}} 
+ a_1\bs{\xi_{n+1}^{(\ell)}} 
+ \bs{l_{n+1}^{(\ell)}}(\bs{{\psi}_{n}^{(\ell)}},\bs{\xi_{n+1}^{(\ell)}})
+ \bs{c} , 
&& \ell=1,\,2,
\end{aligned}
\end{equation}
where $a_0,a_1 \geq 0$, $\bs{c}\in\mathbb{\R}^k$,
$\bs{\xi_{n+1}^{(\ell)}}$ are random variables taking values in
$\{\bs{e_1}, \ldots, \bs{e_k}\}$ and such that
\begin{equation}\label{eq:cond-prob-g}
\begin{aligned}
P\Big(
\bs{\xi_{n+1}^{(\ell)}} = \bs{e_i} \Big| 
\bs{{\psi}_{0}^{(1)}},\, \bs{\xi_{1}^{(\ell)}},\ldots,\bs{\xi_{n}^{(\ell)}}
\Big)& =\\
P\Big(
\bs{\xi_{n+1}^{(\ell)}}
= \bs{e_i} 
\Big| \bs{{\psi}_{n}^{(\ell)}} \Big)
& = {{\psi}_{n\,i}^{(\ell)}},\quad\mbox{for } i= 1, \dots,k,
\end{aligned}
\end{equation}
and $\bs{l_{n+1}^{(\ell)}}$ are measurable functions such that
$|\bs{l^{(\ell)}_{n+1}}|=O(c_{n+1}^{(\ell)})$.  Then, there exist two
stochastic processes $\bs{\widetilde{\psi}^{(\ell)}}=
(\bs{\widetilde{\psi}_{n}^{(\ell)}})_{n\geq 0}$, $\ell=1,2$, evolving
according to the dynamics
\begin{equation}\label{eq:evolution_psi:K_modified-bis}
\begin{aligned}
\bs{\widetilde{\psi}_{n+1}^{(\ell)}} & = a_0 \bs{\widetilde{\psi}_{n}^{(\ell)}} 
+ a_1\bs{\widetilde{\xi}_{n+1}^{(\ell)}} 
+ \bs{l_{n+1}^{(\ell)}}
(\bs{\widetilde{\psi}_{n}^{(\ell)}},\bs{\widetilde{\xi}_{n+1}^{(\ell)}})
+ \bs{c} , 
&& \ell=1,\,2,
\end{aligned}
\end{equation}
with $\bs{\widetilde{\psi}_0^{(\ell)}}=\bs{\psi_{0}^{(\ell)}}$ and
\begin{equation}\label{eq:cond-prob-g-bis}
\begin{aligned}
P\Big(
\bs{\widetilde{\xi}_{n+1}^{(\ell)}} = \bs{e_i} \Big| 
\bs{{\psi}_{0}^{(1)}},\,\bs{{\psi}_{0}^{(2)}}, 
\bs{\widetilde{\xi}_{1}^{(1)}},\bs{\widetilde{\xi}_{1}^{(2)}}
\ldots,\bs{\widetilde{\xi}_{n}^{(1)}}, \bs{\widetilde{\xi}_{n}^{(2)}}
\Big)& =\\
P\Big(
\bs{\widetilde{\xi}_{n+1}^{(\ell)}}
= \bs{e_i} 
\Big| \bs{\widetilde{\psi}_{n}^{(\ell)}} \Big)
& = {\widetilde{\psi}_{n\,i}^{(\ell)}},\quad\mbox{for } i= 1, \dots,k,
\end{aligned}
\end{equation}
and such that, for any $n\geq 0$, we have 
\begin{equation}\label{dis-1}
E\Big[ |
\bs{\widetilde{\psi}_{n+1}^{(2)}} - \bs{\widetilde{\psi}_{n+1}^{(1)}} 
| 
\Big| 
\bs{\widetilde{\psi}_{m}^{(1)}} , \bs{\widetilde{\psi}_{m}^{(2)}} , m\leq n
\Big]
\leq 
(a_0+a_1)
| 
\bs{\widetilde{\psi}_{n}^{(2)}} - \bs{\widetilde{\psi}_{n}^{(1)}} 
| + O(c_{n+1}^{(1)}) + O(c_{n+1}^{(2)}). 
\end{equation} 
\end{theorem}

\begin{remark} \rm 
As a consequence, for each $\ell=1,\,2$, the two stochastic processes
$\bs{\widetilde{\psi}^{(\ell)}}$ and $\bs{\widetilde{\xi}^{(\ell)}}$
have the same joint distribution of $\bs{\psi^{(\ell)}}$ and of
$\bs{\xi^{(\ell)}}$, respectively. Indeed,
$\bs{\widetilde{\psi}^{(\ell)}_0}=\bs{{\psi}^{(\ell)}_0}$ and, by
\eqref{eq:evolution_psi:K_modified}, \eqref{eq:cond-prob-g},
\eqref{eq:evolution_psi:K_modified-bis} and
\eqref{eq:cond-prob-g-bis}, the conditional distributions of
$\bs{\widetilde{\psi}^{\ell}_{n+1}}$ given
$[\bs{\widetilde{\psi}^{(\ell)}_0},\dots,\bs{\widetilde{\psi}^{(\ell)}_n}]$
and of $\bs{\widetilde{\xi}^{\ell}_{n+1}}$ given
$[\bs{\widetilde{\psi}^{(\ell)}_0},\bs{\widetilde{\xi}^{(\ell)}_1}\dots,
    \bs{\widetilde{\xi}^{(\ell)}_n}]$ are the same as the one of
  $\bs{{\psi}^{\ell}_{n+1}}$ given
  $[\bs{{\psi}^{(\ell)}_0},\dots,\bs{{\psi}^{(\ell)}_n}]$ and of
  $\bs{{\xi}^{\ell}_{n+1}}$ given
  $[\bs{{\psi}^{(\ell)}_0}, \bs{{\xi}^{(\ell)}_1}\dots,
      \bs{{\xi}^{(\ell)}_n}]$, respectively.\\ \indent Moreover, from
    inequality \eqref{dis-1}, by recursion, we obtain
\begin{equation}\label{dis-2}
\begin{split} 
E\Big[ | 
\bs{\widetilde{\psi}_{n+1}^{(2)}} - \bs{\widetilde{\psi}_{n+1}^{(1)}} 
|
\Big| 
\bs{\psi_{0}^{(1)}} , \bs{\psi_{0}^{(2)}} 
\Big]
& \leq 
(a_0+a_1)^{n+1} 
| 
\bs{\psi_{0}^{(2)}} - \bs{\psi_{0}^{(1)}} 
| 
\\
& + O\Big(\sum_{j=1}^{n+1} (a_0+a_1)^{n+1-j} (c_{j}^{(1)} + c_{j}^{(2)} )\Big). 
\end{split}
\end{equation} 
\end{remark}

\begin{proof}
We set $\bs{\widetilde{\psi}_{0}^{(\ell)}}= \bs{{\psi}_{0}^{(\ell)}}$, for
$\ell=1,\,2$, and we take a sequence $(U_n)_{n\geq 1}$ of
i.i.d.\ $(0,1)$-uniform random variables, independent of
$\sigma(\bs{{\psi}_{0}^{(1)}},\bs{{\psi}_{0}^{(2)}} )$.
Then, we take the two functions $\bs{h^{(1)}},\bs{h^{(2)}}$ of
Lemma~\ref{lem:Py_coupling} and, for each $\ell$ and any $n\geq 0$, we
recursively define
\begin{align*}
\bs{\widetilde{\xi}_{n+1}^{(\ell)}} & = 
\bs{h^{(\ell)}} 
(\bs{\widetilde{\psi}_{n}^{(1)}} ,\bs{\widetilde{\psi}_{n}^{(2)}} , U_{n+1})
\\
\bs{\widetilde{\psi}_{n+1}^{(\ell)}} & = a_0 \bs{\widetilde{\psi}_{n}^{(\ell)}} 
+ a_1\bs{\widetilde{\xi}_{n+1}^{(\ell)}} 
+
\bs{l_{n+1}^{(\ell)}}
(\bs{\widetilde{\psi}_{n}^{(\ell)}}, \bs{\widetilde{\xi}_{n+1}^{(\ell)}} )
+ 
\bs{c}.
\end{align*}
Setting $\widetilde{\mathcal{F}}_n =
\sigma(\bs{{\psi}_{0}^{(1)}},\bs{{\psi}_{0}^{(2)}}, U_1,
\ldots, U_n)$, we have that $U_{n+1}$ is independent of
$\widetilde{\mathcal{F}}_n$ and, by definition,
$\bs{\widetilde{\xi}_{n}^{(\ell)}}$ and
$\bs{\widetilde{\psi}_{n}^{(\ell)}}$ are
$\widetilde{\mathcal{F}}_n$-measurable, for any $\ell=1,2$ and $n\geq 0$.
Therefore, using relation \eqref{eq:lemcoupl_def}, we get for any
$\ell=1,\,2$, $n\geq 0$ and $i=1,\ldots,k$, 
\begin{equation*}
P\big(
\bs{\widetilde{\xi}_{n+1}^{(\ell)}} 
= \bs{e_i} 
\big| \widetilde{\mathcal{F}}_n \big) = \int
\ind{\{\bs{h^{(\ell)}}
(\bs{\widetilde{\psi}_{n}^{(1)}},\,\bs{\widetilde{\psi}_{n}^{(2)}},\, u) 
= \bs{e_i}\}} \, du
= \widetilde{\psi}_{n\, i}^{(\ell)}.
\end{equation*}
This means that \eqref{eq:evolution_psi:K_modified-bis}, together with
\eqref{eq:cond-prob-g-bis}, holds true.  Finally, by relation
\eqref{eq:lemcoupl_coupl2}, we have
\begin{equation*}
\begin{split}
E\Big[ | 
\bs{{\xi}_{n+1}^{(2)}} - \bs{{\xi}_{n+1}^{(1)}} 
|
\Big| 
\widetilde{\mathcal{F}}_n
\Big]
&=
\int_{(0,1)} 
|\bs{h^{(1)}} (\bs{\widetilde{\psi}^{(1)}_n} ,\bs{\widetilde{\psi}^{(2)}_n} ,\, u)
- 
\bs{h^{(2)}} (\bs{\widetilde{\psi}^{(1)}_n} ,\bs{\widetilde{\psi}^{(2)}_n},\, u) | 
du 
\\
&\leq |\bs{\widetilde{\psi}^{(1)}_n} - \bs{\widetilde{\psi}^{(2)}_n}|
\end{split}
\end{equation*} 
and hence, subtracting \eqref{eq:evolution_psi:K_modified-bis}
with $\ell=2$ from the same relation with $\ell=1$, we obtain
\begin{equation*}
E\Big[ | 
\bs{\widetilde{\psi}_{n+1}^{(2)}} - \bs{\widetilde{\psi}_{n+1}^{(1)}} 
|
\Big| 
\widetilde{\mathcal{F}}_n
\Big]
\leq 
a_0
| 
\bs{\widetilde{\psi}_{n}^{(2)}} - \bs{\widetilde{\psi}_{n}^{(1)}} 
| +
a_1
| 
\bs{\widetilde{\psi}_{n}^{(2)}} - \bs{\widetilde{\psi}_{n}^{(1)}} 
|
+ O(c_{n+1}^{(1)}) + O(c_{n+1}^{(2)}),
\end{equation*} 
and so inequality \eqref{dis-1} holds true.
\end{proof}


\noindent{\bf Acknowledgments}\\

\noindent Giacomo Aletti is a member of the Italian Group ``Gruppo
Nazionale per il Calcolo Scientifico'' of the Italian Institute
``Istituto Nazionale di Alta Matematica'' and Irene Crimaldi is a
member of the Italian Group ``Gruppo Nazionale per l'Analisi
Matematica, la Pro\-ba\-bi\-li\-t\`a e le loro Applicazioni'' of the
Italian Institute ``Istituto Nazionale di Alta Ma\-te\-ma\-ti\-ca''.
\\

\noindent{\bf Funding Sources}\\

\noindent Irene Crimaldi is partially supported by the Italian
``Programma di Attivit\`a Integrata'' (PAI), project ``TOol for
Fighting FakEs'' (TOFFE) funded by IMT School for Advanced Studies
Lucca.  \\



\begin{thebibliography}{10}

\bibitem{AlCrGh}
G.~Aletti, I.~Crimaldi, and A.~Ghiglietti.
\newblock Synchronization of reinforced stochastic processes with a
  network-based interaction.
\newblock {\em Ann. Appl. Probab.}, 27(6):3787--3844, 2017.

\bibitem{ale-cri-ghi-MEAN}
G.~Aletti, I.~Crimaldi, and A.~Ghiglietti.
\newblock Networks of reinforced stochastic processes: asymptotics for the
  empirical means.
\newblock {\em Bernoulli}, 25(4B):3339--3378, 2019.

\bibitem{ale-cri-ghi-WEIGHT-MEAN}
G.~Aletti, I.~Crimaldi, and A.~Ghiglietti.
\newblock Interacting reinforced stochastic processes: Statistical inference
  based on the weighted empirical means.
\newblock {\em Bernoulli}, 26(2):1098--1138, 2020.

\bibitem{AlGhPa}
G.~Aletti, A.~Ghiglietti, and A.~M. Paganoni.
\newblock Randomly reinforced urn designs with prespecified allocations.
\newblock {\em J. Appl. Probab.}, 50(2):486--498, 2013.

\bibitem{AlGhRo}
G.~Aletti, A.~Ghiglietti, and W.~F. Rosenberger.
\newblock Nonparametric covariate-adjusted response-adaptive design based on a
  functional urn model.
\newblock {\em Ann. Statist.}, 46(6B):3838--3866, 2018.

\bibitem{AlGhVi}
G.~Aletti, A.~Ghiglietti, and A.~N. Vidyashankar.
\newblock Dynamics of an adaptive randomly reinforced urn.
\newblock {\em Bernoulli}, 24(3):2204--2255, 2018.

\bibitem{bergh}
D.~Bergh.
\newblock Sample size and chi-squared test of fit--- a comparison between a
  random sample approach and a chi-square value adjustment method using swedish
  adolescent data.
\newblock In Q.~Zhang and H.~Yang, editors, {\em Pacific Rim Objective
  Measurement Symposium (PROMS) 2014 Conference Proceedings}, pages 197--211,
  Berlin, Heidelberg, 2015. Springer Berlin Heidelberg.

\bibitem{BeCrPrRi11}
P.~Berti, I.~Crimaldi, L.~Pratelli, and P.~Rigo.
\newblock A central limit theorem and its applications to multicolor randomly
  reinforced urns.
\newblock {\em J. Appl. Probab.}, 48(2):527--546, 2011.

\bibitem{BeCrPrRi16}
P.~Berti, I.~Crimaldi, L.~Pratelli, and P.~Rigo.
\newblock Asymptotics for randomly reinforced urns with random barriers.
\newblock {\em J. Appl. Probab.}, 53(4):1206--1220, 2016.

\bibitem{BERTONI2018}
D.~Bertoni, G.~Aletti, G.~Ferrandi, A.~Micheletti, D.~Cavicchioli, and
  R.~Pretolani.
\newblock Farmland use transitions after the cap greening: a preliminary
  analysis using markov chains approach.
\newblock {\em Land Use Policy}, 79:789 -- 800, 2018.

\bibitem{cal-che-cri-pam}
G.~Caldarelli, A.~Chessa, I.~Crimaldi, and F.~Pammolli.
\newblock Weighted networks as randomly reinforced urn processes.
\newblock {\em Phys. Rev. E}, 87:020106, Feb 2013.

\bibitem{caron2017}
F.~Caron, W.~Neiswanger, F.~Wood, A.~Doucet, and M.~Davy.
\newblock Generalized p{\'o}lya urn for time-varying pitman-yor processes.
\newblock {\em Journal of Machine Learning Research}, 18(27):1--32, 2017.

\bibitem{chanda}
K.~C. Chanda.
\newblock Chi-squared tests of goodness-of-fit for dependent observations.
\newblock In {\em Asymptotics, Non-Parametrics and Time Series, Statist.
  Textbooks Monogr.}, volume 158, pages 743--756. Dekker, 1999.

\bibitem{chen2013}
M.-R. Chen and M.~Kuba.
\newblock On generalized pólya urn models.
\newblock {\em J. Appl. Probab.}, 50(4):1169--1186, 12 2013.

\bibitem{chessa}
A.~Chessa, I.~Crimaldi, M.~Riccaboni, and L.~Trapin.
\newblock Cluster analysis of weighted bipartite networks: A new copula-based
  approach.
\newblock {\em PLOS ONE}, 9(10):1--12, 10 2014.

\bibitem{collevecchio2013}
A.~Collevecchio, C.~Cotar, and M.~LiCalzi.
\newblock On a preferential attachment and generalized pólya’s urn model.
\newblock {\em Ann. Appl. Probab.}, 23(3):1219--1253, 06 2013.

\bibitem{Cr16}
I.~Crimaldi.
\newblock Central limit theorems for a hypergeometric randomly reinforced urn.
\newblock {\em J. Appl. Probab.}, 53(3):899--913, 2016.

\bibitem{crimaldi-libro}
I.~Crimaldi.
\newblock {\em Introduzione alla nozione di convergenza stabile e sue varianti
  (Introduction to the notion of stable convergence and its variants)},
  volume~57.
\newblock Unione Matematica Italiana, Monograf s.r.l., Bologna, Italy., 2016.
\newblock Book written in Italian.

\bibitem{cri-dai-lou-min}
I.~Crimaldi, P.~Dai~Pra, P.-Y. Louis, and I.~G. Minelli.
\newblock Synchronization and functional central limit theorems for interacting
  reinforced random walks.
\newblock {\em Stochastic Processes and their Applications}, 129(1):70--101,
  2019.

\bibitem{CrDPMi}
I.~Crimaldi, P.~Dai~Pra, and I.~G. Minelli.
\newblock Fluctuation theorems for synchronization of interacting {P}\'{o}lya's
  urns.
\newblock {\em Stochastic Process. Appl.}, 126(3):930--947, 2016.

\bibitem{CriLetPra}
I.~Crimaldi, G.~Letta, and L.~Pratelli.
\newblock {\em A Strong Form of Stable Convergence}, volume 1899, pages
  203--225.
\newblock 2007.

\bibitem{dai-lou-min}
P.~Dai~Pra, P.-Y. Louis, and I.~G. Minelli.
\newblock Synchronization via interacting reinforcement.
\newblock {\em J. Appl. Probab.}, 51(2):556--568, 2014.

\bibitem{Doeb37}
W.~Doeblin and R.~Fortet.
\newblock Sur des cha\^{\i}nes \`a liaisons compl\`etes.
\newblock {\em Bull. Soc. Math. France}, 65:132--148, 1937.

\bibitem{EggPol23}
F.~Eggenberger and G.~P\'olya.
\newblock {\"{U}}ber die statistik verketteter vorg\"ange.
\newblock {\em ZAMM - Journal of Applied Mathematics and Mechanics /
  Zeitschrift f\"ur Angewandte Mathematik und Mechanik}, 3(4):279--289, 1923.

\bibitem{MR0403125}
T.~Gasser.
\newblock Goodness-of-fit tests for correlated data.
\newblock {\em Biometrika}, 62(3):563--570, 1975.

\bibitem{ghiglietti2014}
A.~Ghiglietti and A.~M. Paganoni.
\newblock Statistical properties of two-color randomly reinforced urn design
  targeting fixed allocations.
\newblock {\em Electron. J. Statist.}, 8(1):708--737, 2014.

\bibitem{ghiglietti2017}
A.~Ghiglietti, A.~N. Vidyashankar, and W.~F. Rosenberger.
\newblock Central limit theorem for an adaptive randomly reinforced urn model.
\newblock {\em Ann. Appl. Probab.}, 27(5):2956--3003, 10 2017.

\bibitem{gleser}
L.~J. Gleser and D.~S. Moore.
\newblock The effect of dependence on chi-squared and empiric distribution
  tests of fit.
\newblock {\em The Annals of Statistics}, 11(4):1100--1108, 1983.

\bibitem{Guiv88}
Y.~Guivarc'h and J.~Hardy.
\newblock Th\'eor\`emes limites pour une classe de cha\^\i nes de markov et
  applications aux diff\'eomorphismes d'anosov.
\newblock {\em Annales de l'I.H.P. Probabilit\'es et statistiques},
  24(1):73--98, 1988.

\bibitem{Hairer18}
M.~Hairer.
\newblock Ergodic {P}roperties of {M}arkov {P}rocesses.
\newblock Online at \texttt{http://www.hairer.org/notes/Markov.pdf}.

\bibitem{HallHeyde}
P.~Hall and C.~C. Heyde.
\newblock {\em Martingale limit theory and its application}.
\newblock Academic Press Inc. [Harcourt Brace Jovanovich Publishers], New York,
  1980.
\newblock Probability and Mathematical Statistics.

\bibitem{holmes}
M.~Holmes and A.~Sakai.
\newblock Senile reinforced random walks.
\newblock {\em Stochastic Processes and their Applications}, 117(10):1519 --
  1539, 2007.

\bibitem{ieva}
F.~Ieva, A.~M. Paganoni, D.~Pigoli, and V.~Vitelli.
\newblock Multivariate functional clustering for the morphological analysis of
  electrocardiograph curves.
\newblock {\em Journal of the Royal Statistical Society. Series C (Applied
  Statistics)}, 62(3):401--418, 2013.

\bibitem{IonMar50}
C.~T. Ionescu~Tulcea and G.~Marinescu.
\newblock Th\'{e}orie ergodique pour des classes d'op\'{e}rations non
  compl\`etement continues.
\newblock {\em Ann. of Math. (2)}, 52:140--147, 1950.

\bibitem{knoke2002}
D.~Knoke, G.~W. Bohrnstedt, and A.~Potter~Mee.
\newblock {\em Statistics for Social Data Analysis}.
\newblock F.E.Peacock Publishers, 2002.

\bibitem{lar-pag}
S.~Laruelle and G.~Pag\'es.
\newblock Randomized urn models revisited using stochastic approximation.
\newblock {\em Ann. Appl. Proba.}, 23(4):1409--1436, 2013.

\bibitem{mah}
H.~M. Mahmoud.
\newblock {\em {P}\'olya urn models}.
\newblock Texts in Statistical Science Series. CRC Press, Boca Raton, FL, 2009.

\bibitem{met}
M.~M\'etivier.
\newblock {\em Semimartingales}.
\newblock Walter de Gruyter and Co., Berlin, 1982.

\bibitem{Meyn09}
S.~Meyn and R.~L. Tweedie.
\newblock {\em Markov chains and stochastic stability}.
\newblock Cambridge University Press, Cambridge, second edition, 2009.
\newblock With a prologue by Peter W. Glynn.

\bibitem{AM2019}
A.~Micheletti, G.~Aletti, G.~Ferrandi, D.~Bertoni, D.~Cavicchioli, and
  R.~Pretolani.
\newblock A weighted $\chi^2$ test to detect the presence of a major change
  point in non-stationary {M}arkov chains.
\newblock Submitted for publication, 2019.

\bibitem{Nor72}
M.~F. Norman.
\newblock {\em Markov processes and learning models}.
\newblock Academic Press, New York-London, 1972.
\newblock Mathematics in Science and Engineering, Vol. 84.

\bibitem{MR1894384}
W.~Pan.
\newblock Goodness-of-fit tests for {GEE} with correlated binary data.
\newblock {\em Scand. J. Statist.}, 29(1):101--110, 2002.

\bibitem{PeTaGu08}
Y.~Pei, M.-L. Tang, and J.~Guo.
\newblock Testing the equality of two proportions for combined unilateral and
  bilateral data.
\newblock {\em Communications in statistics - Simulation and computation},
  37(8):1515--1529, 2008.

\bibitem{pemantle2007}
R.~Pemantle.
\newblock A survey of random processes with reinforcement.
\newblock {\em Probab. Surveys}, 4:1--79, 2007.

\bibitem{radlow}
R.~Radlow and E.~F. Alf~Jr.
\newblock An alternate multinomial assessment of the accuracy of the $\chi^2$
  test of goodness of fit.
\newblock {\em Journal of the American Statistical Association},
  70(352):811--813, 1975.

\bibitem{RaoScott81}
J.~N.~K. Rao and A.~J. Scott.
\newblock The analysis of categorical data from complex sample surveys:
  chi-squared tests for goodness of fit and independence in two-way tables.
\newblock {\em J. Amer. Statist. Assoc.}, 76(374):221--230, 1981.

\bibitem{ren}
A.~R{\'e}nyi.
\newblock On stable sequences of events.
\newblock {\em Sankhy\=a Ser. A}, 25:293 302, 1963.

\bibitem{rob}
H.~Robbins and D.~Siegmund.
\newblock A convergence theorem for non negative almost supermartingales and
  some applications.
\newblock In {\em Optimizing Methods in Statistics}, pages 233--257. Academic
  Press, 1971.

\bibitem{sah}
N.~Sahasrabudhe.
\newblock {Synchronization and fluctuation theorems for interacting Friedman
  urns}.
\newblock {\em J. Appl. Probab.}, 53(4):1221--1239, 2016.

\bibitem{SM50}
J.~Sherman and W.~J. Morrison.
\newblock Adjustment of an inverse matrix corresponding to a change in one
  element of a given matrix.
\newblock {\em Ann. Math. Statist.}, 21(1):124--127, 03 1950.

\bibitem{MR3190613}
M.-L. Tang, Y.-B. Pei, W.-K. Wong, and J.-L. Li.
\newblock Goodness-of-fit tests for correlated paired binary data.
\newblock {\em Stat. Methods Med. Res.}, 21(4):331--345, 2012.

\bibitem{tharwat}
A.~Tharwat.
\newblock Independent component analysis: An introduction.
\newblock {\em Applied Computing and Informatics}, 2018.

\bibitem{williams}
D.~Williams.
\newblock {\em Probability with Martingales}.
\newblock Cambridge University Press, 1991.

\bibitem{xu}
D.~Xu and Y.~Tian.
\newblock A comprehensive survey of clustering algorithms.
\newblock {\em Annals of Data Science}, 2(2):165--193, 2015.

\end{thebibliography}

\end{document}